\documentclass[11pt]{article}
\usepackage{graphicx}
\usepackage{amsmath,latexsym,amsbsy,amssymb,amsthm,color,enumitem}

\newtheorem{theorem}{\textbf{Theorem}}[section]

\newtheorem{lemma}[theorem]{\textbf{Lemma}}

\newtheorem{remark}[theorem]{\textbf{Remark}}

\theoremstyle{definition}
\newtheorem{definition}[theorem]{\textbf{Definition}}

\theoremstyle{remark}

\numberwithin{equation}{section}
\allowdisplaybreaks

\usepackage{graphicx}
\font\bigbf=cmbx10 at 16pt

\def\ds{\displaystyle}
\def\L{{\bf L}}

\def\ve{\varepsilon}

\def\R{\mathbb{R}}
\def\N{\mathbb{N}}

\def\vs{\vskip 2em}
\def\vsk{\vskip 4em}

\def\bega{\begin{array}}
\def\enda{\end{array}}
\def\begi{\begin{itemize}}
\def\endi{\end{itemize}}

\def\bel{\begin{equation}\label}
\def\eeq{\end{equation}}
\def\sqr#1#2{\vbox{\hrule height .#2pt
\hbox{\vrule width .#2pt height #1pt \kern #1pt
\vrule width .#2pt}\hrule height .#2pt }}

\begin{document}

\title{\bigbf 
On Kolmogorov entropy compactness estimates for scalar conservation laws without uniform convexity}

\vsk

\author{
Fabio Ancona\footnote{
Dipartimento di Matematica  "Tullio Levi-Civita",
Universit\`a degli Studi di Padova,
Via Trieste 63, 35121 Padova, Italy
} ,
Olivier Glass\footnote{
Ceremade,
Universit\'e Paris-Dauphine, CNRS UMR 7534,
Place du Mar\'echal de Lattre de Tassigny, 
75775 Paris Cedex 16, France
} ,
Khai T. Nguyen\footnote{Department of Mathematics, North Carolina State University, USA
Raleigh, NC 27695, USA}\\
\noalign{\bigskip\medskip}
}

\maketitle

\begin{abstract}
In the case of 
scalar conservation laws
\begin{equation*}
u_{t} + f(u)_{x}~=~0,\qquad t\geq 0, x\in\R,
\end{equation*}
with  uniformly strictly convex flux $f$,
quantitative compactness estimates - in terms of Kolmogorov entropy in ${\bf L}^{1}_{loc}$ -
were established in~\cite{DLG,AON1} for  sets of entropy weak solutions evaluated at a fixed time $t>0$,
whose initial data have a uniformly bounded support and vary in a bounded subset of ${\bf L}^\infty$. 
These estimates reflect the irreversibility features of entropy weak discontinuous solutions of
these nonlinear equations.

We provide here an extension of such estimates  to  the case of scalar conservation laws with a smooth flux function $f$ that either is strictly (but not necessarily uniformly) convex or has a single inflection point with a polynomial degeneracy.
\end{abstract}

\vs
\section{Introduction}

\setcounter{equation}{0} Consider a scalar conservation law in one space dimension
\begin{equation} \label{EqCL}
u_{t} + f(u)_{x}~=~0,
\end{equation}
where $u=u(t,x)$ is the state variable, and $f:\mathbb{R}\to\mathbb{R}$ is a twice continuously differentiable
map.
Without loss of generality, we will suppose
\begin{equation} \label{EqZeroSpeed}
f'(0)=0,
\end{equation}

since one may always reduce the general case to this one by performing
the space-variable and flux transformations $x \to x + t f'(0)$ and $f(u) \to f(u)-u f'(0)$.
It is well known that, no matter how smooth the initial data are, solutions of
the Cauchy problem for~\eqref{EqCL}
generally stay smooth
only up to a critical time beyond which  discontinuities (shocks)
develop. Hence, it is natural to consider weak solutions in the sense of distributions
that, for sake of uniqueness, satisfy an entropy admissibility criterion~\cite{Dafermos:Book}
equivalent to the celebrated Ole\v{\i}nik E-condition~\cite{Oleinik}
which generalizes the classical stability conditions introduced by Lax~\cite{lax57}:

\textbf{Ole\v{\i}nik E-condition.} 
A shock discontinuity located at $x$ and
connecting a left state $u^L\doteq u(t,x-)$ with a right state $u^R\doteq u(t,x+)$ is 
entropy admissible if and only if there holds
\begin{equation}
\label{E-cond}
\frac{f(u^L)-f(u)}{u^L-u}~\geq~\frac{f(u^R)-f(u)}{u^R-u}
\end{equation}
for every $u$ between $u^L$ and $u^R$, where  $u(t,x\pm)$ denote the one-sided limits of $u(t,\cdot)$ at $x$.

The equation~\eqref{EqCL} generates an ${\bf L}^{1}$-contractive semigroup of solutions
$(S_t)_{t \geq 0}$ that associates, to every given initial data $u_{0} \in {\bf L}^{1}(\R) \cap {\bf L}^{\infty}(\R)$,
the unique entropy admissible weak solution $S_t u_0\doteq u(t,\cdot)$ of the corresponding Cauchy problem
(cfr.~\cite{Dafermos:Book, Kruzkov}). This yields the existence of a continuous semigroup $(S_t)_{t \geq 0}$ 
acting on the whole space ${\bf L}^{1}(\R)$.
%
%
Under the assumption that the flux function $f$ is uniformly strictly convex,  
it was shown by Lax~\cite{Lax54} that such a semigroup $S_t$  is
compact as a mapping from ${\bf L}^{1}(\R)$ to ${\bf L}^{1}_{loc}(\R)$,
for every $t>0$.
Indeed, in this case entropy admissible weak solutions satisfy the one-side  Ole\v{\i}nik
inequality~\cite{Oleinik} which yields 
uniform BV-bounds on the solutions at any fixed time $t>0$, which in turn,
applying
Helly's compactness theorem, imply the compactness of the mapping $S_t$.
This property reflects the irreversibility features of entropy weak (discontinuous) solutions
of these equations. De Lellis and Golse~ \cite{DLG}, following a
suggestion by Lax~\cite{Lax78,Lax02}, used the concept of Kolmogorov $\varepsilon$-entropy, recalled below, 
to provide a quantitative estimate of this compactness effect.
\begin{definition}\label{DefKE}
Let $(X,d)$ be a metric space and $K$ a totally bounded subset of $X$. For $\varepsilon>0$, let $\mathcal{N}_{\varepsilon}(K)$ 
be the minimal number of sets in a cover of $K$ by subsets of $X$ having diameter no larger than $2\varepsilon$. Then the $\varepsilon$-entropy of $K$ is defined as
\begin{equation} \nonumber
\mathcal{H}_{\varepsilon}(K \ | \ X) ~\doteq~ \log_{2} \mathcal{N}_{\varepsilon}(K).
\end{equation}
\end{definition}
\noindent
Throughout the paper, we will call an {\it $\varepsilon$-cover}, a cover of $K$ by subsets of $X$ having diameter no larger than $2\varepsilon$. 
\par

In the case of   uniformly strictly convex conservation laws,
De Lellis and Golse established in~\cite{DLG} an upper bound on
the Kolmogorov  {\it $\varepsilon$-entropy}  of the image
set $S_t(\mathcal{C})$  for bounded subsets $\mathcal{C}$ of ${\bf  L}^1$
of order~$1/\varepsilon$. In~\cite{AON1}, we have supplemented  the upper estimate established in~\cite{DLG} with 
a lower bound on the  $\varepsilon$-entropy of~$S_t(\mathcal{C})$ 
of the same order $1/\varepsilon$,
thus showing that the estimate of De~Lellis and Golse was optimal.
Entropy numbers play a central role in various areas of information theory and statistics
as well as of learning theory. In the present setting, this concept  could 
provide a measure of the order of ``resolution'' and of the ``complexity'' of a numerical method for~\eqref{EqCL}, as suggested in~\cite{Lax78,Lax02}.
\par

Aim of this paper is to extend this type of quantitative estimates on the compactness property
of the mapping $S_t$, $t>0$ to the case of conservation laws~\eqref{EqCL}
with a flux function that either is strictly (but non necessarily uniformly) convex or has a single inflection point
and  satisfies some  non-flatness conditions.
Notice that, when one removes the assumption of uniform convexity of the flux function,
entropy weak solutions do not satisfy anymore the classical Ole\v{\i}nik inequality 
and they may have 
unbounded variation (see~\cite{KSC83}).  
However, it was shown in~\cite{KSC86, G08} that for such equations the positive variation of the derivative of the flux
composed with a bounded solution is uniformly bounded at any positive time, hence it belongs to the BV space.
Exploiting this property in the case of a conservation law with a single inflection point,
and invoking~\cite[Theorem~1]{BKP},
given a bounded subsets $\mathcal{C}$ of ${\bf  L}^1$
we first consider  an $\varepsilon'$-covering $\mathcal{U}'$ of 
the set $\mathcal{L}\doteq\big\{f'\circ u \ \big| \ u \in \mathcal{C} \big\}$
with  cardinality $\approx 2^{a/\varepsilon'}$, for some constant $a>0$.
Next, we associate to $\mathcal{U}'$ an $\varepsilon$-covering $\mathcal{U}$ of 
the set $S_t(\mathcal{C})$, with  cardinality $\approx 2^{(a+1)/\varepsilon'}$,
where $\varepsilon' = f'(\varepsilon)$.
As a consequence we find that the 
$\varepsilon$-entropy of~$S_t(\mathcal{C})$  has an upper bound of order $1/f'(\varepsilon)$.
We also show that this estimate is optimal providing a lower bound of the
same order $1/f'(\varepsilon)$
for the  $\varepsilon$-entropy of a subset of $S_t(\mathcal{C})$,
and hence for the  $\varepsilon$-entropy of~$S_t(\mathcal{C})$.
Namely, performing a similar analysis as in~\cite{AON1},
we establish such a lower bound for the  $\varepsilon$-entropy 
of~$S_t(\mathcal{C}^+\cup \mathcal{C}^-)$, where $\mathcal{C}^+$, $\mathcal{C}^-$, denote the classes of initial data
in $\mathcal{C}$ which assume only nonnegative and nonpositive values, respectively.
Notice that, 
for the
particular class of fluxes $f(u)=u^{m+1}/(m+1)$, $m$ even, we find that 
the Kolmogorov $\varepsilon$-entropy of~$S_t(\mathcal{C})$ is
of order $1/\varepsilon^m$, which shows how accurate this concept is
in reflecting the nonlinearity of the flux.
We finally prove that even in the case of  strictly, but not uniformly, convex flux
there hold the same upper and lower bounds of order $1/f'(\varepsilon)$
for the Kolmogorov $\varepsilon$-entropy of~$S_t(\mathcal{C})$.
\par

Specifically, we shall assume that the flux function satisfies one
of the {\bf standing assumptions}:
\begin{itemize}
\item[\ \  {\bf (C)}] $f:\R\to\R$ is a twice continuously differentiable, strictly convex function.

\item[\ \  {\bf (NC)}] $f:\R\to\R$ is a smooth, non convex function with
a single inflection point at zero having polynomial degeneracy, i.e. such that
\begin{equation}
\label{hyp-NC}
\begin{gathered}
f^{(j)}(0)~=~0\quad \text{for all} \quad j =2,\dots, m,\qquad f^{(m+1)}(0)~\neq~ 0\,,
\\
\noalign{\medskip}
f''(u)\cdot u \cdot \mathrm{sign} \big(f^{(m+1)}(0)\big)~>~0\qquad\forall~u\in \R\,\setminus \{0\}\,,
\end{gathered}
\end{equation}
for 
some even integer $m\in\N\setminus\{0\}$.
\end{itemize}
Notice that, generically, smooth fluxes satisfy one of the assumptions {\bf (C), (NC)}, since
a generic property of smooth maps $f:\R\to\R$ is that $f^{(3)}(x)\neq 0$ whenever $f''(x)=0$.

In connection with a flux $f:\R\to\R$ and any constant $M>0$, we introduce a map \linebreak $\Delta_{f,M}:(0,+\infty)\to\R$ measuring
the oscillation of $f'$, defined by setting
\begin{equation}
\label{delta-funct-def}
\Delta_{_{f,M}}(s)~\doteq ~s \cdot
\inf_{\substack{\\|u|,|v|\leq M\ u\cdot v~\geq~ 0\\
|v-u|\geq s}}
\bigg|\frac{f'(v)-f'(u)}{v-u}\bigg|
\qquad\forall~s>0\,.
\end{equation}
Notice that since in~\eqref{delta-funct-def} we are taking the infimum in a compact subset of $\R^2$,
if $f$ satisfies either of the assumptions  {\bf (C)} or  {\bf (NC)}, it follows that
$\Delta_{_{f,M}}(s)>0$ for all $s>0$.

We then consider  sets of bounded, compactly supported initial data of the form
\begin{equation} \label{DefCLM}
{\mathcal C}_{[L,M]}~\doteq~\Big\{ u_{0} \in {\bf L}^{\infty}(\R) 
\ \big| \ 
\mbox{Supp\,}(u_{0}) \subset [-L,L]\ , \ 
\| u_{0} \|_{{\bf L}^{\infty}} \leq M
\Big\}.
\end{equation}
%
The main results of the paper show that the Kolmogorov $\varepsilon$-entropy of $S_t(C_{[L,M]})$
with respect to the  ${\bf L}^1$-topology
is of order $\approx \varepsilon^{-m}$ for fluxes
satisfying the assumption {\bf (NC)},
and has an upper bound of order
$\approx (\Delta_{_{f,M}}(\varepsilon))^{-1}$ for fluxes satisfying the assumption {\bf (C)}. 
Precisely, we prove the following upper and lower bounds for the Kolmogorov $\varepsilon$-entropy of $S_t(C_{[L,M]})$
\begin{theorem}\label{ThmLUBC}
Let $f:\mathbb{R}\to\mathbb{R}$ be a function
satisfying~\eqref{EqZeroSpeed} and  the  assumption {\bf (C)}, and
let $\{S_t\}_{t\geqslant 0}$ be the semigroup of entropy weak solutions generated by~\eqref{EqCL}
on the domain ${\bf L}^1(\R)$. 
Then, given  $L,M,T>0$,  
for every $\varepsilon>0$ sufficiently small
the following estimates hold:
\begin{align}
\label{uest-c}
\mathcal{H}_{\varepsilon}\Big(S_T({\mathcal C}_{[L,M]}) \ | \ {\bf L}^{1}(\R)\Big) &\leq ~
\Gamma_{\!1}^+\cdot\frac{1}{\Delta_{_{f,M}}\big(\frac{\varepsilon}{\gamma_1^+}\big)}\,,
%
\\
\noalign{\medskip}
\label{lest-c}
\mathcal{H}_{\varepsilon}\Big(S_T({\mathcal C}_{[L,M]}) \ | \ {\bf L}^{1}(\R)\Big) &\geq ~
\Gamma_{\!1}^-
\cdot\frac{1}{\varepsilon\cdot
\min\bigg\{\displaystyle{\max_{z\in \big[0,\, \frac{6\varepsilon}{L}\big]}} f''(z),\,\displaystyle{\max_{z\in \big[
-\frac{6\varepsilon}{L},\, 0\big]}} f''(z)\bigg\}}\,,
\end{align}
where
%
\begin{align}
\label{G1+-def}
\Gamma_{\!1}^+&=~
c_1 \bigg(L+T+\frac{L^2}{T}\bigg),\qquad\quad \gamma_1^+~=~c_1 \Big(1+L+T\Big)\,,
%
\\
\noalign{\medskip}
\Gamma_{\!1}^-&=~\frac{c_1}{T}\,,
\end{align}
%
for some constant $c_1>0$ depending only on $f$ and $M$.
\end{theorem}
\begin{remark}
\rm
In the case where 
the derivative $f'$ of a  strictly convex flux $f$ is a convex function
on~$[0,+\infty)$ and a concave function on $(-\infty, 0]$,
and we assume that~\eqref{EqZeroSpeed} holds,
by definition~\eqref{delta-funct-def} it follows that 
\begin{equation*}
\Delta_{_{f,M}}(s)~=~\min\big\{|f'(-s)|,\, |f'(s)|\big\}\qquad\forall~s>0\,,
\end{equation*}
for every $M>0$, while
\begin{equation*}
\min\Big\{\displaystyle{\max_{z\in [0, s]}} f''(z),\,\displaystyle{\max_{z\in [-s, 0]}} f''(z)\Big\}~=~
\min\big\{f''(-s),\, f''(s)\big\}\qquad\forall~s>0\,.
\end{equation*}
Therefore, in this case, by~\eqref{uest-c}-\eqref{lest-c} the  ${\bf L}^1$-Kolmogorov $\varepsilon$-entropy of $S_t(C_{[L,M]})$
turns out to be of or\-der~$\approx 1/|f'(\pm\varepsilon)|\approx1/\big(\varepsilon\cdot f''(\pm\varepsilon)\big)$. Instead, if we assume
that
$f''(u)\geq c>0$ for all $u\in \R$,
applying the mean-value theorem to $f'$ it follows that
\begin{equation*}
\Delta_{_{f,M}}(s)~\geq~ c\cdot s\qquad\quad\forall~s>0\,.
\end{equation*}
On the other hand, 
for every fixed $M>0$, there exists some constant  $c_M>0$
such that
\begin{equation*}
\min\Big\{\displaystyle{\max_{z\in [0, s]}} f''(z),\,\displaystyle{\max_{z\in [-s, 0]}} f''(z)\Big\}
~\leq~ c_M
\quad\qquad\forall~s\in (0, M]\,.
\end{equation*}
%
Thus, in this second case
we recover the estimate  $\mathcal{H}_{\varepsilon} (S_T({\mathcal C}_{[L,M]}) \, | \, {\bf L}^{1}(\R))\approx 1/\varepsilon$
established in~\cite{AON1,DLG} for uniformly strictly convex fluxes.
\end{remark}
\begin{remark}
\rm
If we consider a smooth, strictly convex flux $f$ 
with a polynomial degeneracy at zero, i.e. such that
\begin{equation}
\label{hyp-C-pol}
\begin{gathered}
f^{(j)}(0)~=~0\quad \text{for all} \quad j =1,\dots, m,\qquad f^{(m+1)}(0)~\neq~0\,,
\\
\noalign{\medskip}
f''(u)
~>~0\qquad\forall~u\in \R\,\setminus \{0\}\,,
\end{gathered}
\end{equation}
for some odd integer $m\in\N$, one can show that there exist some constant $\alpha_M>0$
depending on $f, M$, and $\overline\alpha>0$ depending only on $f$, such that
\begin{equation}
\label{second-der-bound-1}
\Delta_{_{f,M}}(s)~\geq~\frac{s^m}{\alpha_M},
\qquad\quad
\min\Big\{\displaystyle{\max_{z\in [0, s]}} f''(z),\,\displaystyle{\max_{z\in [-s, 0]}} f''(z)\Big\}
~\leq~\overline\alpha \cdot s^{m-1}\,,
\end{equation}
for all $s>0$ sufficiently small
(see Remark~\ref{pol-c} and Lemma~\ref{up-bounds-second-der-f}).
Hence, for fluxes satisfying the assumption~\eqref{hyp-C-pol}, 
by~\eqref{uest-c}-\eqref{lest-c} the  ${\bf L}^1$-Kolmogorov $\varepsilon$-entropy of $S_t(C_{[L,M]})$
turns out to be of order~$\approx 1/\varepsilon^m$.
\end{remark}
\begin{theorem}\label{ThmLUBNC}
Let $f:\mathbb{R}\to\mathbb{R}$ be a function
satisfying~\eqref{EqZeroSpeed} and  the  assumption {\bf(NC)}.
Then, in the same setting of Theorem~\ref{ThmLUBC},
for any given $L,M,T>0$, and for every $\varepsilon>0$ sufficiently small,
the following estimates hold:
\begin{align}
\label{uest-nc}
\mathcal{H}_{\varepsilon}\Big(S_T({\mathcal C}_{[L,M]}) \ | \ {\bf L}^{1}(\R)\Big) &\leq 
\Gamma_{\!2}^+\cdot\frac{1}{\varepsilon^m}\,,
%
\\
\noalign{\medskip}
\label{lest-nc}
\mathcal{H}_{\varepsilon}\Big(S_T({\mathcal C}_{[L,M]}) \ | \ {\bf L}^{1}(\R)\Big) &\geq ~
\Gamma_{\!2}^-\cdot\frac{1}{\varepsilon^m}\,,
\end{align}
where
\begin{equation}
\begin{aligned}
\Gamma_{\!2}^+&=
c_2 \bigg(1+L+T+\frac{L^2}{T}\bigg)^{m+1}
\\
\noalign{\medskip}
\Gamma_{\!2}^-&=~c_2\cdot\frac{L^{m+1}}{T}\,,
\end{aligned}
\end{equation}
for some constant $c_2>0$ depending only on $f$ and $M$.
\end{theorem}
\begin{remark}
\rm
If a flux $f$ satisfies the assumption {\bf(NC)} and~\eqref{EqZeroSpeed},
one can show
that, 
for every fixed $M>0$, there exist some constant $\beta_M>0$
such that
\begin{equation}
\frac{s^m}{\beta_M}~\leq~\Delta_{_{f,M}}(s)
~\leq~\beta_M\cdot {s^m}
\end{equation}
for all $s>0$ sufficiently small (see Lemma~\ref{nc-flux-prop-1}). Hence, the estimates on the  Kolmogorov $\varepsilon$-entropy of $S_t(C_{[L,M]})$
provided by Theorem~\ref{ThmLUBNC} are of the same order as the ones stated in Theorem~\ref{ThmLUBC}.
\end{remark}
We observe that, for fluxes having one inflection point where all derivatives vanishes, 
the composition of the derivative of the flux with the solution of~\eqref{EqCL} fails in general to belong to the BV space
(see~\cite{Marc} and Remark~\ref{rem:not-pd} here). 
However, for {\it weakly genuinely nonlinear} fluxes, that is to say for fluxes with no affine parts, it is shown in~\cite[Theorem 26]{Ta}
that equibounded sets of entropy solutions of~\eqref{EqCL} are still relatively compact in ${\bf L}^1$
(see also~\cite{Marc}).
Therefore, for fluxes of such class that do not fulfill the assumption {\bf(NC)},
it remains an open problem to provide quantitative compactness estimates on the solutions set of~\eqref{EqCL}.
In this case, a different approach from the one developed in the
the present paper must be pursued to obtain upper bounds on the Kolmogorov $\varepsilon$-entropy of 
the solution set, perhaps exploiting the 
$\mathrm{BV}^\Phi$-bounds obtained in~\cite[Theorem 1]{Marc}, $\Phi$ being a convex function
linked to the degeneracy of the flux.

The paper is organized as follows.  In Section~\ref{sec:not} we collect notations
and preliminary results concerning the theory of scalar conservation laws and 
the estimates of the Kolmogorov \linebreak $\varepsilon$-entropy for sets of functions with
uniformly bounded variation. In Section~\ref{sec:uce} we establish the upper bounds 
on the   $\varepsilon$-entropy of 
the solution set stated in Theorems~\ref{ThmLUBC}-\ref{ThmLUBNC},
while the proof of the lower bounds is carried out in Section~\ref{sec:lce}.

\section{Notations and preliminaries}
\label{sec:not} Throughout the paper we shall denote by
\begin{itemize}
\item ${\bf L}^1(\R)$, the Lebesgue space of all (equivalence classes of)
summable functions on $\R$, equipped with the usual norm $\|\cdot\|_{\bf L^1}$;
\item ${\bf L}^{\infty}(\R)$, the space of all essentially bounded functions on $\R$, equipped with the usual norm $\|\cdot\|_{\bf L^{\infty}}$;
\item $Supp(u)$, the essential support of a function $u\in {\bf L^\infty}(\R)$;
\item $TV\{u\,|\, D\}$, the total variation of $u$ on the interval $D\subset\R$; in the case where $D=\R$ we just write $TV\{u\}$;
\item $BV(D)$, the set of functions with bounded total variation on $D$;
\item $\lfloor x \rfloor\doteq \max\big\{z\in \mathbb{Z}\, | z\leq x\big\}$, the integer part of $x$.
\end{itemize}
\begin{remark} \label{Rem:RightAndLeftLimits}
\rm
We recall~\cite{Dafermos:Book, Kruzkov} that a scalar conservation law~\eqref{EqCL} generates a 
unique ${\bf L}^1$-contractive semigroup 
$\left\{S_t:{\bf L^1}(\R)\to{\bf L}^{1}(\R)\right\}_{t\geq 0}$ 
that associates to any $u_0\in {\bf L}^1(\R)\cap{\bf L}^{\infty}(\R)$ the unique entropy solution 
\[
u(t,x)~\doteq~S_t u_0(x)\qquad\qquad x\in\R, t>0\,,
\]
of (\ref{EqCL}) with initial data $u(0,x)=u_0$.
Notice that, if the the flux function $f$ satisfies either of the assumptions {\bf (C)} or {\bf(NC)}
stated in the Introduction, although $S_t u_0$ may well have unbounded variation, it is still
true that $S_t u_0$ admits one-sided limits $S_t u_0(\pm x) $ at every point $x\in\R$.
This is the consequence of the  Lax-Ole\v{\i}nik representation formula~\cite{lax57} 
in the  {\bf (C)} case, and of the $\text{BV}^{\frac{1}{p}}$ regularity (see~\cite[Theorem 3]{Marc})
in the {\bf(NC)} case.
\end{remark}
%
%
For any $L,M>0$, consider the class of functions in~\eqref{DefCLM}
and set
\begin{equation}
\label{der-flux-bound}
f'_M~\doteq ~ \sup_{|v|\leq M}~|f'(v)|\,.
\end{equation}
The next classical result provides an upper bound on the ${\bf L}^{\infty}$-norm and on the support of $S_T u_0$
for every $u_0\in\mathcal{C}_{[L,M]}$.
\begin{lemma}
\label{L1} 
Let $f:\R\to\R$ be a differentiable map.
For any $L, M, T>0$ and $u_0\in\mathcal{C}_{[L,M]}$, there holds 
\begin{equation}
\label{linf-supp-nbounds-1}
\big\|S_T u_0\big\|_{{\bf L}^{\infty}(\R)}~\leq~M\qquad\mathrm{and}\qquad Supp(S_T u_0)~\subseteq~\big[\!- l_{[L,M,T]},\, l_{[L,M,T]}\big]
\end{equation}
with 
\begin{equation}
\label{LT-def}
l_{[L,M,T]}~\doteq~L+T\cdot f'_M\,.
\end{equation}
Moreover, if $u_0\in\mathcal{C}_{[L,M]}\cap  BV(\R)$, then one has $S_T u_0\in BV(\R)$.
\end{lemma}
\begin{proof} The monotonicity of the solution operator $S_t$ yields~\cite{Dafermos:Book, Kruzkov}:
\begin{equation}
\label{linf-bound1}
\big\|S_t u_0\big\|_{{\bf L}^{\infty}(\R)}~\leq~\big\|u_0\big\|_{{\bf L}^{\infty}(\R)}~\leq~M\qquad\forall~t\geq 0\,.
\end{equation}
Next observe that,  for any $u_0\in\mathcal{C}_{[L,M]}$, we can find a sequence $\{u^\nu\}_\nu$,
$u^\nu\in BV(\R)$
with $Supp(u^\nu) \subset Supp(u_0)$, such that $u^\nu\to u_0$ in ${\bf L}^1$.
This, in turn, 
implies that 
\begin{equation}
\label{conv-ft-appr}
S_t u^\nu\quad\overset{{\bf L}^1}\longrightarrow\quad S_t u_0
\qquad\forall~\nu\,,\ \forall~t>0\,.
\end{equation}
Moreover, recalling that $S_t u^\nu$ can be obtained as limit of piecewise constant front tracking
approximations~\cite[Chapter 6]{Bressan}, we deduce that
\bel{Spt-t}
Supp (S_t u^\nu)~\subseteq~\big[\!- l_{[L,M,t]},\, l_{[L,M,t]}\big]\qquad \mathrm{with}\qquad  l_{[L,M,t]}~\doteq~L+t\cdot f'_M
\qquad\forall~\nu\,,\ \forall~t>0\,,
\eeq
with $f'_M$ as in~\eqref{der-flux-bound}.
Thus, \eqref{linf-bound1}-\eqref{Spt-t}
together yield~\eqref{linf-supp-nbounds-1}. The a-priori bounds on the total variation of the solution guarantee also
that $S_T u_0\in BV(\R)$ whenever $u_0\in BV(\R)$ (see~\cite[Theorem 6.1]{Bressan}.
\end{proof}
We next collect the uniform upper bounds on the total variation of the flux of an entropy weak solutions
established in~\cite[Theorem 3.4, Theorem 4.9]{KSC86} (see also~\cite{dafermos:gc}, \cite[Section 11.2]{Dafermos:Book},
 \cite[Theorem 2]{Marc}). 
%
%
\begin{lemma}
\label{L2} 
Assume that $f:\R\to\R$ is a  function satisfying 
either of the {\bf (C)} or {\bf (NC)} conditions.
Then, for any $L,M,T>0$ and for every $u_0\in\mathcal{C}_{[L,M]}$, there holds
\bel{TV-cv-ncv1}
TV\big\{f'\circ S_Tu_0~|~\R\,\big\}~\leq~C_1\,\bigg(1+\frac{L}{T}\bigg),
\eeq
for some constant $C_1>0$ depending only on $f$ and $M$.
\end{lemma}
\begin{proof} 
For convenience of the reader we provide a sketch of the proof since the constants in the right-hand side of~\eqref{TV-cv-ncv1} slightly differs form the ones
in the cited references.

{\bf 1.} Assume that $f$ satisfies the {\bf (C)} condition. \\
Observe first that, 
because of the non intersection property of minimal and maximal backward characteristics~\cite{dafermos:gc}, 
one deduces a one-sided Lipschitz condition on the derivative of the flux~\cite[Section 11.2]{Dafermos:Book}:
\begin{equation}
\label{1side-lip-der-flux}
f'\big(S_T u_0(y)\big)-f'\big(S_Tu_0(x)\big)~\leq~{y-x\over T}\qquad\forall~x<y\,.
\end{equation}
On the other hand, by Lemma~\ref{L1}  we have
$Supp(S_T(u_0))\subseteq [-l_{[L,M,T]},\, l_{[L,M,T]}]$.
Thus, since~\eqref{1side-lip-der-flux}
implies that $x\to f'\big(S_Tu_0(x)\big)-\frac{x}{T}$ is a non increasing map, we find
\begin{multline*}
TV\big\{f'\circ S_Tu_0~|~\R\,\big\}~=~\lim_{\ve\to 0}~TV\big\{f'\circ S_Tu_0~|~(-l_{[L,M,T]}-\!\ve,\,l_{[L,M,T]}+\ve)\big\}\cr
~\leq~ \lim_{\ve\to 0}~\Big[TV\Big\{f'\circ S_Tu_0-{\cdot \over T}~\Big|~(-l_{[L,M,T]}-\!\ve,\, l_{[L,M,T]}+\ve)\Big\}+\\
+TV\Big\{{\cdot \over T}~|~(-l_{[L,M,T]}-\!\ve,\,l_{[L,M,T]}+\ve)\Big\}\Big]~\leq~{4l_{[L,M,T]}\over T}\,,
\end{multline*}
which, by definition~\eqref{LT-def}, yields (\ref{TV-cv-ncv1}).
\medskip

{\bf 2.} Assume that $f$ satisfies the {\bf (NC)} condition. \\
Since by Lemma~\ref{L1}  we have
$\|S_T u_0\|_{{\bf L}^{\infty}(\R)} \leq M$, 
$Supp(S_T(u_0))\subseteq [-l_{[L,M,T]},\, l_{[L,M,T]}]$, 
invoking \cite[Theorem 4.9]{KSC86} (see also~\cite[Theorem 2]{Marc}) we deduce
that, for every $\ve>0$, there holds
\begin{equation}
\label{TV-ncv1}
TV\big\{f'\circ S_Tu_0~|~(-l_{[L,M,T]}-\ve,l_{[L,M,T]}+\ve)\big\}~\leq~\frac{C_M\cdot 2(l_{[L,M,T]}+\ve)}{T}+\widetilde C_M,
\end{equation}
where $C_M, \widetilde C_M>0$
are constants depending only on the flux $f$ and on $M$.
Hence, relying on~\eqref{LT-def}, \eqref{TV-ncv1} we derive
\begin{eqnarray*}
TV\{f'\circ S_T(u_0)~|~(-\infty,+\infty)\}&=&\lim_{\ve\to 0}~TV\{f'\circ S_T(u_0)~|~(-l_{[L,M,T]}-\ve,l_{[L,M,T]}+\ve)\}\\
&\leq&\frac{8C^{f,M}_1L}{T}+2\cdot \Big(C_M\cdot f'_M+\widetilde C_M\Big)\,,
\end{eqnarray*}
which yields (\ref{TV-cv-ncv1}).
\end{proof}
\begin{remark}
\label{rem:not-pd}
\rm
In the non convex case a bound as in~\eqref{TV-cv-ncv1} in general does not hold without
the assumption of polynomial degeneracy in~\eqref{hyp-NC}. 
In fact, it has been exhibited in~\cite[Section 8.1]{Marc} an example of
a flux $f(u)$ having
one inflection point at zero, with $f^j(0)=0$ for all $j\in\N$, $j\geq 2$,
and of an initial data $u_0\in {\bf L}^{\infty}(\R)$ with compact support, such that
$f'\circ S_t(u_0)\notin BV(\R)$ for almost every $t$ in an interval of $(0, \infty)$.
\end{remark}

To complete this section, we recall now two results that provide an upper bound on the  $\ve$-entropy for sets of functions
with uniformly bounded total variation and a lower bound for the
$\ve$-entropy of sets of functions having uniformly bounded one-side derivative.
\begin{lemma}{\rm (\cite[Theorem 1]{BKP})}
\label{Co-BV} Given $L,V>0$, consider the set
\begin{equation}
\label{FLVset-def}
\mathcal{F}_{[L,V]}~\doteq~\Big\{g:\R\to [-V,V]~\big|~\big|~Supp(g)\subseteq [-L,L], \ TV\{g\}~\leq~2V\Big\}\,.
\end{equation}
Then, for all $\ve\leq {VL\over 3}$, there holds
\[
\mathcal{H}_{\ve}\big(\mathcal{F}_{[L,V]}~|~{\bf L}^1(\R)\big)~\leq~48VL\cdot\frac{1}{\ve}\,.
\] 
Moreover, there exists a set of piecewise constant functions $\{g_1,\dots,g_p\} \subset \mathcal{F}_{[L,V]}$,
with 
\begin{equation*}
p~\leq~\Bigg\lfloor2^{\big(\frac{48 VL}{\varepsilon}\big)}\Bigg\rfloor+1\,,
\end{equation*}
that satisfy
\begin{equation*}
g_i(x)~=~g_i\Big(\!-L+\frac{2L}{N}\cdot\nu\Big)\qquad\quad
\forall~x\in \Big[\!-L+\frac{2L}{N}\cdot\nu,\ -L+\!\frac{2L}{N}\cdot(\nu\!+\!1)\Big),
\qquad\nu\in\{0,1,\dots,N\!-\!1\},
\end{equation*}
with
\begin{equation*}
N~\geq~\bigg\lfloor\frac{8LV}{\varepsilon}\bigg\rfloor\,,
\end{equation*}
and such that
\begin{equation}
\mathcal{F}_{[L,V]}~\subset~\bigcup_{i=1}^p B\big(g_i,\,\ve\big)\,,
\end{equation}
where $B\big(g_i,\,\ve\big)$ denotes the ${\bf L}^1(\R)$-ball centred at $g_i$ of radius $\ve$.
\end{lemma} 
\medskip 
\begin{lemma}{\rm (\cite[Proposition 2.2]{AON1})}
\label{lowb} 
Given $L,h,b>0$, consider the sets
\begin{equation}
\label{Bpm}
\begin{aligned}
\mathcal{B}_{[L,h, \leq b]} &\doteq~\Big\{v\in BV(\R)~\big|~Supp(v)\subseteq [-L,L], \ \|v\|_{{\bf L}^{\infty}(\R)}\leq h, \ Dv\leq b
\Big\}\,,
\\
\noalign{\medskip}
\mathcal{B}_{[L,h, \geq -b]} &\doteq~\Big\{v\in BV(\R)~\big|~Supp(v)\subseteq [-L,L], \ \|v\|_{{\bf L}^{\infty}(\R)}\leq h, \ Dv\geq -b
\Big\}\,,
\end{aligned}
\end{equation}
where the inequalities $Dv\leq b$, $Dv\geq -b$
must  be understood in the sense of measures, i.e. the Radon measure $D v$
satisfies $D v (J)\leq b\cdot|J|$,  $D v (J)\geq -b\cdot|J|$, respectively,
for every Borel set $J\subset\R$, $|J|$ being the Lebesgue measure of $J$.
Then, for any $0<\ve\leq {Lh\over 6}$, there holds 
\begin{equation}
\nonumber
\begin{aligned}
\mathcal{H}_{\ve}\big(\mathcal{B}_{[L,h,\geq b]}~\big|~{\bf L}^1(\R)\big)~&\geq~ {2bL^2\over 27\ln 2}\cdot \frac{1}{\ve}\,,
\\
\noalign{\smallskip}
\mathcal{H}_{\ve}\big(\mathcal{B}_{[L,h,\leq -b]}~\big|~{\bf L}^1(\R)\big)~&\geq~ {2bL^2\over 27\ln 2}\cdot \frac{1}{\ve}\,.
\end{aligned}
\end{equation}
\end{lemma}
\section{Upper compactness estimates} 
\label{sec:uce}
We derive in this section 
upper bounds on the $\ve$-entropy in $\L^1$ of
 $S_T(\mathcal{C}_{[L,M]})$ for the class of initial data $\mathcal{C}_{[L,M]}$ in~\eqref{DefCLM},
 when the flux function $f$ satisfies either of the assumptions {\bf (C)} or {\bf(NC)}
 stated in the Introduction. 

Towards a proof of~\eqref{uest-c}, \eqref{uest-nc},  we first establish an upper bound on the $\ve$-entropy in~$\L^1$ of
the set  
 \begin{equation}
 \label{LLMT-set-def}
 \mathcal{L}_{[L,M,T]}~\doteq~ 
 \Big\{ f'\circ u
\ \big| \ u \in S_T({\mathcal C}_{[L,M]})\Big\}\,.
 \end{equation}
\begin{lemma}
\label{entr-f'-sol-set}
In the same setting of Theorem~\ref{ThmLUBC}
or Theorem~\ref{ThmLUBNC},
assume that $f:\R\to\R$ is a  function satisfying 
either of the {\bf (C)} or {\bf (NC)} conditions and
that~\eqref{EqZeroSpeed} holds.
Then, given any $L,M,T>0$, for all $\ve\leq \frac{\Gamma_{\!1}^+}{288}$, there holds
\begin{equation}
\label{uest-fprime-c}
\mathcal{H}_{\varepsilon}\Big( \mathcal{L}_{[L,M,T]}
\ | \ {\bf L}^{1}(\R)\Big)~\leq~ 
\frac{\Gamma_{\!1}^+}{2}\cdot\frac{1}{\varepsilon}\,,
\end{equation}
with $\Gamma_{\!1}^+$ as in~\eqref{G1+-def}.
Moreover, there exists a set of piecewise constant functions
$\{g_1,\dots,g_p\} $,
with 
\begin{equation*}
p~\leq~\Bigg\lfloor2^{\Big(\frac{\Gamma_{\!1}^+}{\small\strut 2\,\varepsilon}\Big)}\Bigg\rfloor+1\,,
\end{equation*}
that enjoy the following properties:
\begin{itemize}
\item[(i)]
For any $i=1,\dots, p$, one has 
\[
Supp(g_i)~\subseteq~\big[\!- l_{[L,M,T]},\, l_{[L,M,T]}\big],
\]
\[
Im(g_i)~\subseteq~\big[\!-\!f'_M, f'_M\big]\quad\text{if}\quad \text{\bf (C)} \ \ holds,
\quad\qquad Im(g_i)~\subseteq~\big[0, f'_M\big]\quad\text{if}\quad \text{\bf (NC)} \ \ holds\,,
\]
and
\begin{equation*}
\hspace{-20pt}
g_i(x)~=~g_i(x_\nu)
\qquad\quad\forall~x\in[x_\nu,\,x_{\nu+1})\,,
\qquad \nu\in\{0,1,\dots,N\!-\!1\},
\end{equation*}
with
\begin{equation*}
x_{\nu}~\doteq~ -l_{[L,M,T]}+ {2\,  l_{[L,M,T]}\over N}\cdot\nu\,,
\qquad
\nu\in \{0,1,..., N\}\,,
\quad\quad
N~\geq~\bigg\lfloor\frac{8\, l_{[L,M,T]}\cdot V_{[L,M,T]}}{\varepsilon}\bigg\rfloor,
\end{equation*}
where $f'_M$, $l_{[L,M,T]}$ are the constants defined
in~\eqref{der-flux-bound}, \eqref{LT-def}, respectively, and 
\begin{equation}
\label{VLT-def}
V_{[L,M,T]}~\doteq~\max
\bigg\{
\frac{C_1}{2}\cdot\bigg(1+\frac{L}{T}\bigg),\, f'_M
\bigg\},
\end{equation}
$C_1$ being the constants defined in~\eqref{TV-cv-ncv1}.

\item[(ii)] 
\begin{equation*}
\mathcal{L}_{[L,M,T]}~\subset~\bigcup_{i=1}^p B\big(g_i,\,\ve\big)\,.
\end{equation*}
\end{itemize}
\end{lemma}
\begin{proof}
Observe first that, relying on Lemmas~\ref{L1}-\ref{L2}, we have
\begin{equation}
\label{derf-sol-incl1}
\mathcal{L}_{[L,M,T]}~\subseteq~
\mathcal{F}_{[l_{[L,M,T]},\,V_{[L,M,T]}]}\,,
\end{equation}
where $\mathcal{F}_{[l_{[L,M,T]},\,V_{[L,M,T]}]}$ is a set defined as in~\eqref{FLVset-def}.
Therefore, invoking Lemma~\ref{Co-BV}, we derive
\begin{equation}
\begin{aligned}
\mathcal{H}_{\varepsilon}\Big(\mathcal{L}_{[L,M,T]}
\ | \ {\bf L}^{1}(\R)\Big) 
&\leq~ 
\mathcal{H}_{\varepsilon}\Big(\mathcal{F}_{[l_{[L,M,T]},\,V_{[L,M,T]}]} \ | \ {\bf L}^{1}(\R)\Big) \\
\noalign{\smallskip}
&\leq~ 48\cdot
\max
\bigg\{
\frac{C_1\big(L+T\cdot f'_M)}{2}\cdot\bigg(1+\frac{L}{T}\bigg),\, f'_M \big(L+T\cdot f'_M)
\bigg\}\cdot\frac{1}{\varepsilon}\,,
\end{aligned}
\end{equation}
which yields~\eqref{uest-fprime-c}, and we deduce
the existence of piecewise constant functions $\{g_1,\dots,g_p\}$
enjoying the properties (i)-(ii).
\end{proof}
\medskip

\subsection{Strictly (not necessarily uniformly) convex fluxes} 
In this subsection, we will study the case where $f$ is a convex function
satisfying the assumption {\bf (C)} which in particular implies that   $f'$ is strictly increasing
and hence invertible on $\R$. 

In order to establish~\eqref{uest-c}, we will use 
the following technical lemma providing
an estimate of the ${\bf L}^1$-distance of two elements $u,v\in S_T(\mathcal{C}_{[L,M]})$ in terms
of the  ${\bf L}^1$-distance of $f'\circ u, f'\circ v$. 
To this end, consider the map
\begin{equation}
\label{delta+-funct-def}
\widehat\Delta_{_{f,M}}(s)~\doteq~s \cdot
\inf_{\substack{\\
|u|,|v|\leq M
\\
|v-u|\geq s}}
\bigg|\frac{f'(v)-f'(u)}{v-u}\bigg|
\qquad\forall~s>0\,,
\end{equation}
which differs form the map in~\eqref{delta-funct-def} for the fact that the infimum is taken also over 
pairs $u,v$ of opposite sign.
Observe that the maps 
$s\mapsto \widehat\Delta_{_{f,M}}(s)$, $s\mapsto \frac{\widehat\Delta_{_{f,M}}(s)}{s}$ are strictly increasing and thus invertible. Moreover,  one has
\begin{equation}
\label{inv-der-f-lip-est-0}
\hspace{1in}\widehat\Delta_{_{f,M}}\big(|u-v|\big)~\leq~ |f'(u)-f'(v)|
\quad\qquad\forall~u,v\in\R\quad\text{s.t.}\quad |u|, |v|\leq M\,.
\end{equation}
\begin{lemma}
\label{inv-der-f-lip-est}
Let $f:\R\to\R$ be a differentiable map.
Given any $L,M>0$, for every $u,v\in {\bf L}^\infty(\R)$ with 
\begin{equation}
\label{hyp-inv-der-f-lip-est}
\|u\|_{{\bf L}^{\infty}}~\leq~M, \ \|v\|_{{\bf L}^{\infty}}~\leq~M
\qquad\mathrm{and}\qquad Supp(u)\subset [-L,L], \  Supp(v)\subset [-L,L],
\end{equation}
there holds
\begin{equation}
\label{inv-der-f-lip-est-1}
\big\|u-v\big\|_{{\bf L}^1(\R)}
~\leq~(1+2L)\cdot {\widehat\Delta}^{\strut -1}_{_{f,M}}\Big(\big\|f'\circ u-f'\circ v\big\|_{{\bf L}^1(\R)}\Big).
\end{equation}
\end{lemma}
\begin{proof}\hspace{1in}

\noindent
{\bf 1.} We claim that, setting 
\begin{equation}
\label{ro-def}
\rho~\doteq~{\widehat\Delta}^{\strut -1}_{_{f,M}}\Big(\big\|f'\circ u-f'\circ v\big\|_{{\bf L}^1(\R)}\Big), 
\end{equation}
one has
\begin{equation}
\label{E1}
\big|u(x)-v(x)\big|~\leq~\rho\cdot\max\bigg\{1,\ \frac{\big|f'(u(x))-f'(v(x))\big|}{\ \ \ \ \ \big\|f'\circ u-f'\circ v\big\|_{{\bf L}^1(\R)}}
\bigg\}\qquad\forall~x\in\R\,.
\end{equation}
Indeed, assume that $|u(x)-v(x)|\geq\rho$. Then, relying on~\eqref{inv-der-f-lip-est-0}, \eqref{ro-def},
and on the monotonicity of $s\to \frac{\widehat\Delta_{_{f,M}}(s)}{s}$, we estimate
\begin{equation}
\begin{aligned}
\big|u(x)-v(x)\big| &\leq~ \frac{|u(x)-v(x)|}{\widehat\Delta_{_{f,M}}\big(|u(x)-v(x)|\big)}\cdot\big|f'(u(x))-f'(v(x))\big|
\\
&\leq~ \frac{\rho}{\widehat\Delta_{_{f,M}}(\rho)}\cdot\big|f'(u(x))-f'(v(x))\big|
\\
&=~\rho\cdot \frac{\big|f'(u(x))-f'(v(x))\big|}{\ \ \ \ \ \big\|f'\circ u-f'\circ v\big\|_{{\bf L}^1(\R)}}
\end{aligned}
\end{equation}
which yields~\eqref{E1}.
\medskip

\noindent{\bf 2.}
Thanks to~\eqref{ro-def}, \eqref{E1}, and since by~\eqref{hyp-inv-der-f-lip-est} one has $u=v=0$ on $\R\setminus [-L,L]$, 
we derive
\begin{equation}
\begin{aligned}
\big\|u-v\big\|_{{\bf L}^1(\R)} &\leq~\int_{-L}^L \big|u(x)-v(x)\big|dx \\
\noalign{\smallskip}
&\leq~\int_{-L}^L\rho \cdot\Bigg(1+ \frac{\big|f'(u(x))-f'(v(x))\big|}{\ \ \ \ \ \big\|f'\circ u-f'\circ v\big\|_{{\bf L}^1(\R)}}\Bigg)dx\\
\noalign{\smallskip}
&\leq~2L \,\rho +\rho\,,
\end{aligned}
\end{equation}
which proves~\eqref{inv-der-f-lip-est-1}.
\end{proof}
The next lemma shows that  $\Delta_{_{f,M}}, \widehat\Delta_{_{f,M}}$
are comparable maps. 
\begin{lemma}
\label{delta-maps-prop}
Given a map $f:\R\to\R$ satisfying the assumption {\bf (C)},
let $\Delta_{_{f,M}}, \widehat\Delta_{_{f,M}}$
be the maps defined in~\eqref{delta-funct-def}, \eqref{delta+-funct-def}, respectively.
Then, one has
\begin{equation}
\label{Delta-comp}
\Delta_{_{f,M}}(s/2)~\leq~{\widehat\Delta}_{_{f,M}}(s)~\leq~\Delta_{_{f,M}}(s)
\end{equation}
\end{lemma}
\begin{proof}
The second inequality in~\eqref{Delta-comp} is an immediate consequence of the 
definitions~\eqref{delta-funct-def}, \eqref{delta+-funct-def}. Towards a proof of the first inequality in~\eqref{Delta-comp},
given $u\leq 0\leq v$, relying on the monotonicity of $f'$ we find
\begin{equation}
\label{Dhat-est1}
\begin{aligned}
\bigg| \frac{f'(v)-f'(u)}{v-u}  \bigg|& = ~
\frac{f'(v)-f'(0)}{v-u} +\frac{f'(0)-f'(u)}{v-u}
\\
\noalign{\smallskip}
&=~\left(\frac{v}{v-u}\right)\cdot  
\frac{f'(v)-f'(0)}{v}
+\left(\frac{-u}{v-u}\right)\cdot 
\frac{f'(0)-f'(u)}{-u}\,.
\end{aligned}
\end{equation}
Therefore, observing that $v-u\geq s$ implies $\max\{v,-u\}\geq s/2$, we deduce from~\eqref{Dhat-est1} that
for all $-M\leq u\leq 0\leq v\leq M$, such that $v-u\geq s$, there holds
\begin{equation}
\label{Dhat-est2}
\begin{aligned}
\bigg| \frac{f'(v)-f'(u)}{v-u}  \bigg|&\geq~ 
\begin{cases}
\dfrac{1}{2}\cdot  
\dfrac{f'(v)-f'(0)}{v}
\quad&\text{if}\qquad v=\max\{v,-u\}\,,
%
\\
\\
\dfrac{1}{2}\cdot  
\dfrac{f'(0)-f'(u)}{-u}
\quad&\text{if}\qquad -u=\max\{v,-u\}\,,
\end{cases}
\\
\noalign{\smallskip}
&\geq~\frac{1}{2}\cdot D(s/2)
\end{aligned}
\end{equation}
where
\begin{equation}
\label{D-def}
D(s)~\doteq~\inf_{\substack{\\
|u|,|v|\leq M,\,u\cdot v\geq 0 \ \\
|v-u|\geq s}}
\bigg|\frac{f'(v)-f'(u)}{v-u}\bigg|\,.
\end{equation}
Taking the infimum in the left-hand side of~\eqref{Dhat-est2} over all 
$-M\leq u\leq 0\leq v\leq M$ with $v-u\geq s$, we thus find
\begin{equation}
\label{Dhat-est3}
\widehat D(s)~\geq~  \frac{1}{2}\cdot D(s/2)
\qquad\forall~s>0\,,
\end{equation}
where
\begin{equation}
\label{Dhat-def}
\widehat D(s)~\doteq~\inf_{\substack{\\
|u|,|v|\leq M \ \\
|v-u|\geq s}}
\bigg|\frac{f'(v)-f'(u)}{v-u}\bigg|\,.
\end{equation}
Then, observing that by~\eqref{delta-funct-def}, \eqref{delta+-funct-def}, we have
\begin{equation}
\label{Delta-Deltahat-D-Dhat}
\Delta_{_{f,M}}(s)~=~s \cdot D(s),\qquad\qquad
\widehat\Delta_{_{f,M}}(s)~=~s \cdot \widehat D(s)\qquad\quad\forall~s>0\,,
\end{equation}
we recover from~\eqref{Dhat-est3} the first inequality in~\eqref{Delta-comp}.
\end{proof}

We are now ready to provide the:

\medskip

{\bf Proof of upper bound~\eqref{uest-c} of Theorem~\ref{ThmLUBC}}\\

By virtue of Lemma~\ref{entr-f'-sol-set}, given any
\begin{equation}
\label{eps-hyp-1}
0~<~\varepsilon< \big(1+2\, l_{[L,M,T]}\big) \cdot {\widehat\Delta}^{\strut -1}_{_{f,M}}\Big( \frac{\Gamma_{\!1}^+}{124}\Big),
\end{equation}
with $l_{[L,M,T]}$ as in~\eqref{LT-def}, and setting
\begin{equation}
\label{epsprime-def}
\varepsilon'~\doteq~  {\widehat\Delta}_{_{f,M}}\bigg(\frac{\varepsilon}{1+2\, l_{[L,M,T]}}\bigg)\,,
\end{equation}
there holds 
\begin{equation}
\mathcal{N}_{\varepsilon'}\Big(\mathcal{L}_{[L,M,T]}
\ | \ {\bf L}^{1}(\R)\Big)~\leq~2^{\Big(\frac{\Gamma_{\!1}^+}{\small\strut 2\, \varepsilon'}\Big)}  \,.
\end{equation}
%
Therefore, 
there exists a set of functions
\begin{equation}
\big\{g_1,\dots,g_p\big\}
~\subset~\mathcal{L}_{[L,M,T]}\,,
\end{equation}
with
\begin{equation}
\label{p-bound1}
p~\leq~\Bigg\lfloor2^{\Big(\frac{\Gamma_{\!1}^+}{\small\strut 2\,\varepsilon'}\Big)}\Bigg\rfloor+1\,,
\end{equation}
such that 
\begin{equation}
\label{derf-solset-incl-1}
\mathcal{L}_{[L,M,T]}~\subseteq~\bigcup_{i=1}^p B\big(g_i,\,\ve'\big)\,,
\end{equation}
where $B\big(g_i,\,\ve'\big)$ denotes the ${\bf L}^1(\R)$-ball centred at $g_i$ of radius $\ve'$.
Notice that, by Lemma~\ref{L1} and because of~\eqref{EqZeroSpeed}, we have 
\[
\mathcal{L}_{[L,M,T]}~\subseteq~{\mathcal C}_{[l_{[L,M,T]},f'_M]}\,.
\]
Hence~\eqref{derf-solset-incl-1} yields
\begin{equation}
\label{derf-solset-incl-2}
\mathcal{L}_{[L,M,T]}~\subseteq~{\mathcal C}_{[l_{[L,M,T]},f'_M]} \cap \bigcup_{i=1}^p B\big(g_i,\,\ve'\big)\,.
\end{equation}
On the other hand, observing that by~\eqref{EqZeroSpeed} one has
\begin{equation}
g~\in~{\mathcal C}_{[l_{[L,M,T]},f'_M]}\qquad\Longrightarrow\qquad (f')^{-1}\circ g~\in~{\mathcal C}_{[l_{[L,M,T]}, M]}\,,
\end{equation}
and because of~\eqref{epsprime-def}, invoking Lemma~\ref{inv-der-f-lip-est} we deduce that
for all $i=1,\dots,p$, there holds
\begin{equation}
\label{derf-solset-incl-3}
g~\in~ {\mathcal C}_{[l_{[L,M,T]},f'_M]}, \ \|g-g_i\|_{{\bf L}^1(\R)}~<~\ve' \qquad\Longrightarrow\quad
\big\|(f')^{-1}\circ g-(f')^{-1}\circ g_i\big\|_{{\bf L}^1(\R)}<\ve\,.
\end{equation}
Hence, we deduce from~\eqref{derf-solset-incl-2}, \eqref{derf-solset-incl-3} that
\begin{equation}
\label{derf-solset-incl-4}
\begin{aligned}
S_T({\mathcal C}_{[L,M]})&\subseteq~ \bigcup_{i=1}^p 
\Big\{(f')^{-1}\circ g \ \big| \ g\in  {\mathcal C}_{[l_{[L,M,T]},f'_M]}\cap B\big(g_i,\,\ve'\big)  \Big\}
\\
\noalign{\smallskip}
&\subseteq~
\bigcup_{i=1}^p B\big((f')^{-1}\circ g_i,\,\ve\big)\,.
\end{aligned}
\end{equation}
Thus, for all $\ve>0$ satisfying~\eqref{eps-hyp-1}, we have produced an $\ve$-cover of $S_T({\mathcal C}_{[L,M]})$ in ${\bf L}^1$ of cardinality $p$ which, thanks to~\eqref{Delta-comp}, \eqref{p-bound1}, is bounded by
\begin{equation}
\label{card-est-1}
p~\leq~1+ 2^{\Big(\frac{\Gamma_{\!1}^+}{\small\strut 2\,\varepsilon'}\Big)}
~\leq~ 2^{\Big(\frac{\Gamma_{\!1}^+}{\small\strut \varepsilon'}\Big)}
~=~ 2^{\Big(\frac{\Gamma_{\!1}^+}{\small\strut {\widehat\Delta}_{_{f,M}}(2\varepsilon/\gamma_1^+)}\Big)}
~\leq~ 2^{\Big(\frac{\Gamma_{\!1}^+}{\small\strut {\Delta}_{_{f,M}}(\varepsilon/\gamma_1^+)}\Big)}
\end{equation}
with $\gamma_1^+ \doteq 2(1+2\, l_{[L,M,T]})$ as in~\eqref{G1+-def} because of~\eqref{LT-def}.
Taking the base-2 logarithm in~\eqref{card-est-1} we then derive the estimate~\eqref{uest-c}.
\qed

\subsection{Fluxes with one inflection point having polynomial degeneracy} 
In this subsection we will assume that $f$ is a non convex function 
satisfying the assumption~{\bf (NC)} and~\eqref{EqZeroSpeed}.
To fix the ideas we shall consider the case where  $f^{(m+1)}(0)> 0$, the case with $f^{(m+1)}(0)< 0$
being entirely similar.
Therefore, throughout this subsection we shall assume that, for some even integer $m\in\N\setminus\{0\}$,
there holds
\begin{equation}
\label{hyp-NC2}
\begin{gathered}
f^{(j)}(0)~=~0\quad \text{for all} \quad j =1,\dots, m,\qquad f^{(m+1)}(0)~>~ 0\,,
\\
\noalign{\medskip}
f''(u)\cdot u ~>~0\qquad\forall~u\in \R\,\setminus \{0\}\,.
\end{gathered}
\end{equation}
This implies that  the function $f'$ is strictly decreasing on $(-\infty,0]$ and strictly increasing on~$[0,+\infty)$. Moreover,  $f'$ is positive on $\R\setminus\{0\}$.

Towards a proof of~\eqref{uest-nc}  we first establish some technical lemmas
concerning the flux $f$ and the  function $\Delta_{_{f,M}}$ defined in~\eqref{delta-funct-def},
and providing bounds on the ${\bf L}^1$-distance of two elements $u,v\in S_T(\mathcal{C}_{[L,M]})$ in terms
of the  ${\bf L}^1$-distance of $f'\circ u, f'\circ v$.

\begin{lemma}
\label{nc-flux-prop-1}
Let $f:\R\to\R$ be a smooth map satisfying the assumption~\eqref{hyp-NC2}.
%
%
For any $M>0$, there exist constants $\kappa_{M}\in (0,1)$, $\beta_M, \sigma_M>0$ depending only on $f$ and $M$, such that 
the following hold.
\begin{equation}
\label{der-f-est-1}
\begin{aligned}
\big|f'(u)-f'(u/2)\big|&~\geq~\kappa_M\cdot |f'(u)|
\\
\noalign{\smallskip}
\big|f'(u/2)\big|&~\geq~\kappa_M\cdot |f'(u)|
\end{aligned}
\qquad\quad\forall~u\in[-M,M]\,,
\end{equation}
\begin{equation}
\label{der-f-est-2}
\sup_{u\in [-M,M]\setminus\{0\}}~\left|\left\{{f(u)-f(0)\over uf'(u)} \right\}\right|~\leq~1-{\kappa_M\over 2}~<~1\,,
\end{equation}
\medskip
\begin{equation}
\label{der-f-est-3}
\frac{s^m}{\beta_M}~\leq~ \Delta_{_{f,M}}(s)
~\leq~ \beta_M\cdot {s^m}\qquad\quad\forall~s\in(0, \sigma_M]\,.
\end{equation}
%
\end{lemma}
\begin{proof}\hspace{1in}

\noindent
{\bf 1.} 
Observe first that, by the monotonicity property of $f'$ and since $f'$ is always non negative,
the inequalities in~\eqref{der-f-est-1} are equivalent to
\begin{equation}
\label{der-f-est-1b}
\begin{aligned}
f'(u)-f'(u/2)&~\geq~\kappa_M\cdot f'(u)
\\
\noalign{\smallskip}
f'(u/2)&~\geq~\kappa_M\cdot f'(u)
\end{aligned}
\qquad\quad\forall~u\in[-M,M]\,.
\end{equation}
Next, by writing a Taylor approximation of the derivative of the flux in the origin and relying on~\eqref{EqZeroSpeed}, 
we find
\begin{equation}
\label{der-f-est-4}
\begin{aligned}
\frac{f'(u)}{2}-f'(u/2)&=~\frac{f^{(m+1)}(0)}{m!} \Big(\frac{u^{m}}{2}-(u/2)^{m}\Big)+u^{m}\cdot o(1)
\\
\noalign{\smallskip}
&=~u^{m}\cdot \Bigg( \frac{f^{(m+1)}(0)}{m!}\Big((1/2)-(1/2)^{m}\Big)+o(1)\Bigg)\,,
\end{aligned}
\end{equation}
and
\begin{equation}
\label{der-f-est-4b}
\begin{aligned}
f'(u/2)-\frac{f'(u)}{2^{m+1}}&= \frac{f^{(m+1)}(0)}{m!} \Big((u/2)^{m}-\frac{1}{2}(u/2)^{m}\Big)+u^{m}\cdot o(1)
\\
\noalign{\smallskip}
&=u^{m}\cdot \Bigg( \frac{f^{(m+1)}(0)}{m!}\frac{1}{2^{m+1}}
+o(1)\Bigg)\,,
\end{aligned}
\end{equation}
where $o(1)$ denotes a function converging to zero when $u\to 0$. Since $f^{m+1}(0)> 0$
and $m$ is even,
we deduce from~\eqref{der-f-est-4} that there will be some constant $u_0>0$
such that
\begin{equation}
\label{der-f-est-5}
\begin{aligned}
f'(u)-f'(u/2) &\geq \frac{1}{2} f'(u) 
%
\\
\noalign{\smallskip}
f'(u/2) &\geq \frac{1}{2^{m+1}}f'(u) 
\end{aligned}
\qquad\quad\forall~u\in [-u_0, u_0]\,.
\end{equation}
On the other hand, setting
\begin{equation}
\label{c-prime-second-def}
c_0\doteq \inf_{u_0\leq |u|\leq M}
f'(u)-f'(u/2)\,,
\qquad\quad
c'_0\doteq \inf_{u_0\leq |u|\leq M}
f'(u/2)\,,
\qquad\quad
\widehat c_0\doteq \sup_{u_0\leq |u|\leq M}
f'(u)\,,
\end{equation}
we find
\begin{equation}
\label{der-f-est-1c}
\begin{aligned}
f'(u)-f'(u/2)~&\geq~\frac{c_0}{\widehat c_0}\cdot f'(u)
\\
\noalign{\smallskip}
f'(u/2)~&\geq~\frac{c'_0}{\widehat c_0}\cdot f'(u)
\end{aligned}
\qquad\quad\forall~u\in[-M,M]\setminus [-u_0, u_0]\,,
\end{equation}
where $c_0, c'_0, \widehat c_0$ are  positive constants since in~\eqref{c-prime-second-def} we are taking the infimum and the supremum
of positive continuous functions on a compact subset of $\R$. Hence, \eqref{der-f-est-5}, \eqref{der-f-est-1c} together
yield \eqref{der-f-est-1b} taking
\[
\kappa_M\doteq \min\bigg\{\frac{1}{2^{m+1}},\,\frac{c_0}{\widehat c_0},\,\frac{c'_0}{\widehat c_0}\bigg\}\,.
\]
\medskip

\noindent{\bf 2.}
Notice that condition~\eqref{der-f-est-1b} implies
\[
f'(u/2)~\leq~(1-\kappa_M)\cdot f'(u)\qquad\forall u~\in [-M,M]\,.
\] 
Hence, relying on the non negativity and monotonicity property  of $f'$,  for any $u\in  [-M,M]\setminus\{0\}$
we derive the estimate:
\begin{equation*}
\begin{aligned}
\left|{f(u)-f(0)\over u}\right|
&=~\left|{\ds \int^u_0}f'(s)~ds\over u\right|~\leq~\left|{\ds \int^{u/2}_0f'(s)~ds+\int^{u}_{u/2}}f'(s)~ds\over u\right|
\\
\noalign{\smallskip}
&\leq~\frac{1}{2}\Big(f'(u/2)+f'(u)\Big)
\leq~\left(1-{\kappa_M\over 2}\right)\cdot f'(u)\,,
\end{aligned}
\end{equation*}
which yields (\ref{der-f-est-2}).

\medskip

\noindent{\bf 3.}
In order to establish~\eqref{der-f-est-3}, it will be sufficient to show that
there exist constants $s_o, k_0>0$ such that
%
there holds
\begin{equation}
\label{D-est-1}
\qquad\quad
\frac{1}{k_0}\cdot \frac{\min\big\{f'(-s/2),\,f'(s/2)\big\}}{s/2}~\leq~D(s)~\leq~k_0\cdot \frac{\min\big\{f'(-s),\,f'(s)\big\}}{s}
\qquad\ \forall~s\in(0, s_0]\,,
\end{equation}
with $D$ as in~\eqref{D-def},
since then one recovers~\eqref{der-f-est-3} from~\eqref{D-est-1} 
recalling~\eqref{Delta-Deltahat-D-Dhat} and taking the Taylor expansion of $f'$
at zero.
\smallskip

Towards a proof of~\eqref{D-est-1}, observe first that by writing the 
Taylor expansion of $f^{(3)}$ at zero we find
\begin{equation}
\label{der-3-f-taylor}
f^{(3)}(u)~=~u^{m-2}\cdot \Bigg( \frac{f^{(m+1)}(0)}{(m-2)!}+o(1)\Bigg)
\end{equation}
where $o(1)$ denotes a function converging to zero when $u\to 0$. Since $f^{(m+1)}(0)> 0$
and $m$ is even,
we deduce from~\eqref{der-3-f-taylor} that there will be some constant $u'_0\in (0, M)$
such that 
\begin{equation*}
\label{der-3-f-taylor-2}
f^{(3)}(u)~\geq~0
\qquad\quad\forall~u\in [-u'_0, u'_0]\,,
\end{equation*}
which in turn implies that $f'$ is a convex map on $[-u'_0, u'_0]$.
Therefore, recalling that $f'(0)=0$, we deduce that
\begin{equation}
\label{D-prop1}
\inf_{\substack{\\
|u|,|v|\leq u'_0,\,u\cdot v\geq 0 \ \\
|v-u|\geq s}}
\bigg|\frac{f'(v)-f'(u)}{v-u}\bigg|
~=~\frac{\min\big\{f'(-s),\,f'(s)\big\}}{s}
\quad\qquad\forall~s\in (0,u'_0]\,.
\end{equation}
Since by definition~\eqref{D-def} we have
\begin{equation}
D(s)~\leq~  \inf_{\substack{\\
|u|,|v|~\leq~ u'_0,\,u\cdot v~\geq~ 0 \ \\
|v-u|\geq s}}
\bigg|\frac{f'(v)-f'(u)}{v-u}\bigg|\qquad\forall~s\,,
\end{equation}
we obtain from~\eqref{D-prop1} the upper bound in~\eqref{D-est-1} with $s_0=u'_0$, $k_0=1$.

Concerning the lower bound in~\eqref{D-est-1}, applying the mean-value theorem to $f'$ we find
\begin{equation}
\label{D-est-2}
\inf_{\substack{\\
u'_0\leq |u|,|v|\leq M\\ 
u\cdot v\geq 0\ \ }}
\bigg|\frac{f'(v)-f'(u)}{v-u}\bigg|~\geq~ 
c''_0
\end{equation}
where
\begin{equation}
\label{f-der-2-bound}
c''_0~\doteq~\inf_{u'_0\leq |u|\leq M} \big|f''(u)\big|\,.
\end{equation}
Here, $c''_0$ is a positive  constant since in~\eqref{f-der-2-bound} we are taking the infimum 
of a  continuous function on a compact subset of $\R\setminus\{0\}$,
which  is positive on $\R\setminus\{0\}$ because of~\eqref{hyp-NC2}.
On the other hand, observing that by~\eqref{hyp-NC2} we have $\lim_{s\to 0} \frac{f'(|s|)}{s}=f''(0)=0$,  
it follows that 
\begin{equation}
\label{D-est-3}
\frac{\max\big\{f'(-s),\,f'(s)\big\}}{s}~<~c''_0
\quad\qquad\forall~s\in (0,u''_0]\,,
\end{equation}
for some constant $u''_0\in (0, u'_0)$. Therefore, by virtue of~\eqref{D-prop1}, \eqref{D-est-2}, \eqref{D-est-3},  
we derive
\begin{equation}
\label{D-est-0}
D(s)=\min\left\{\frac{f'(-s)}{s},\,\frac{f'(s)}{s},\,  
\inf_{\substack{\\
|u|\leq u'_0\leq |v|\leq M \ \\
u\cdot v\geq 0,\ |v-u|\geq s}}
\bigg|\frac{f'(v)-f'(u)}{v-u}\bigg|
\right\}
\quad\qquad\forall~s\in (0,u''_0]\,.
\end{equation}
In order to provide a lower bound for
\begin{equation*}
\inf_{\substack{\\
|u|\leq u'_0\leq |v|\leq M \ \\
u\cdot v\geq 0,\ |v-u|\geq s}}
\bigg|\frac{f'(v)-f'(u)}{v-u}\bigg|
\end{equation*}
we shall consider the case where $0\leq u\leq u'_0\leq v\leq M$.
Relying on the monotonicity of $f'$ on~$[0,+\infty)$, on convexity of $f'$ on~$[-u'_0, u'_0]$ 
and on~\eqref{f-der-2-bound},we find
\begin{equation}
\label{D-est-4}
\begin{aligned}
\bigg| \frac{f'(v)-f'(u)}{v-u}  \bigg|&=~ \frac{f'(v)-f'(u'_0)}{v-u} +\frac{f'(u'_0)-f'(u)}{v-u}
\\
\noalign{\smallskip}
&\geq~
\left(\frac{v-u'_0}{v-u}\right)\cdot c''_0+ \left( \frac{u'_0-u}{v-u}\right)\cdot \frac{f'(u'_0-u)}{u'_0-u}\,.
\end{aligned}
\end{equation}
We now distinguish two cases:
\begin{itemize}
\item[(i)] If $u'_0-u>\frac{s}{2}$, then it follows from~\eqref{D-est-4} that
\begin{equation}
\label{D-est-5}
\bigg| \frac{f'(v)-f'(u)}{v-u}  \bigg|~\geq~\min
\left\{c''_0,\,  \frac{f'(u'_0-u)}{u'_0-u}\right\}~\geq~ 
\min
\left\{c''_0,\,  \frac{f'(s/2)}{s/2}\right\}\,.
\end{equation}

\item[(ii)] If $u'_0-u\leq \frac{s}{2}$, then one has $u'_0-u\leq  \frac{v-u}{2}$ which implies 
$\frac{v-u'_0}{v-u}\geq \frac{1}{2}$. Hence, we deduce from~\eqref{D-est-4} that
\begin{equation}
\label{D-est-6}
\bigg| \frac{f'(v)-f'(u)}{v-u}  \bigg|~\geq~
\frac{c''_0}{2}\,.
\end{equation}
\end{itemize}
Therefore, by virtue of~\eqref{D-est-3}, \eqref{D-est-5}, \eqref{D-est-6}, and relying again on the convexity of $f'$ on~$[-u'_0, u'_0]$,
we find
\begin{equation}
\label{D-est-7}
\inf_{\substack{\\
0\leq u\leq u'_0\leq v\leq M \ \\
|v-u|\geq s}}
\bigg|\frac{f'(v)-f'(u)}{v-u}\bigg|~\geq~\frac{1}{2}\cdot  \frac{f'(s/2)}{s/2}
 \quad\qquad\forall~s\in (0,u''_0]\,.
\end{equation}
The case where $-M\leq v\leq -u'_0\leq u\leq 0$ can be treated in an entirely similar way.
Hence, \eqref{D-est-0}, \eqref{D-est-7} together yield
the lower bound in~\eqref{D-est-1} with $s_0=u''_0$, $k_0=2$, thus completing the proof of the Lemma.
\end{proof}
\begin{remark}
\label{pol-c}
\rm
If we consider a smooth, convex flux satisfying the assumption~\eqref{hyp-C-pol},
with the same arguments of the proof of Lemma~\ref{nc-flux-prop-1} one can show that
the same type of lower bound in~\eqref{der-f-est-3} holds. In fact, assume to fix the ideas that 
$f^{(m+1)}(0)> 0$. Then, given $M>0$, relying on~~\eqref{hyp-C-pol}, \eqref{der-3-f-taylor} one deduces that there exist constants
$\widetilde u'_0>\widetilde u''_0>0$ such that:
\begin{itemize}
\item[(i)] $f'$ is a convex map on $[0, \widetilde u'_0]$
and a concave map on $[-\widetilde u'_0, 0]$;
\item[(ii)]
\begin{equation}
\label{D-est-2-c}
\inf_{\substack{\\
\widetilde u'_0\leq |u|,|v|\leq M\\ 
u\cdot v\geq 0\ \ }}
\bigg|\frac{f'(v)-f'(u)}{v-u}\bigg|~\geq~ 
\widetilde c''_0~\doteq~ 
\inf_{\widetilde u'_0\leq |u|\leq M} \big|f''(u)\big|>0\,;
\end{equation}
\item[(iii)]
\begin{equation}
\label{D-est-3-c}
\frac{\max\big\{|f'(-s)|,\,|f'(s)|\big\}}{s}~<~\widetilde c''_0
\quad\qquad\forall~s\in (0,\widetilde u''_0]\,.
\end{equation}
\end{itemize}
By virtue of (i), (ii), (iii), one then finds that
\begin{equation}
\label{D-est-0-c}
D(s)~=~\min\left\{\frac{|f'(-s)|}{s},\,\frac{|f'(s)|}{s},\,  
\inf_{\substack{\\
|u|\leq \widetilde u'_0\leq |v|\leq M \ \\
u\cdot v\geq 0,\ |v-u|\geq s}}
\bigg|\frac{f'(v)-f'(u)}{v-u}\bigg|
\right\}
\quad\qquad\forall~s\in (0,\widetilde u''_0]\,.
\end{equation}
where $D(s)$ is defined as in~\eqref{D-def}. On the other hand, 
relying on the monotonicity of $f'$ and on (i), (ii), (iii), we derive
as in the proof of  of Lemma~\ref{nc-flux-prop-1} that
\begin{equation}
\label{D-est-7-c}
\inf_{\substack{\\
0\leq u\leq \widetilde u'_0\leq v\leq M \ \\
|v-u|\geq s}}
\bigg|\frac{f'(v)-f'(u)}{v-u}\bigg|~\geq~\frac{1}{2}\cdot  \frac{|f'(s/2)|}{s/2}
 \quad\qquad\forall~s\in (0,\widetilde u''_0]\,.
\end{equation}
Thus, \eqref{D-est-0-c}, \eqref{D-est-7-c} together yield the lower bound
\begin{equation}
\Delta_{_{f,M}}(s)~\geq~\frac{s^m}{\alpha_M}
\quad\qquad\forall~s\in (0,\widetilde u''_0]\,,
\end{equation}
for some constant $\alpha_M>0$.
\end{remark}
\medskip
\begin{lemma}\label{k-est}
Let $f:\R\to\R$  be a smooth map satisfying the assumption~\eqref{hyp-NC2}.
Given any  $L,M,T>0$, for every $u\in S_T({\mathcal C}_{[L,M]})$,
and for any $x<y$ such that 
\begin{equation}
\label{sign-hyp}
\mathrm{sign} (u(x))~\neq~\mathrm{sign} (u(y)),
\end{equation}
there holds
\bel{TV-e}
TV\big\{f'\circ u\ \big|\  [x,y]\big\}~\geq~ \widetilde{\kappa}_M\cdot \max \big\{f'(u(x)), f'(u(y))\big\}\,,
\eeq
for some constant $\widetilde \kappa_M\in(0,1)$  depending only on $f$ and $M$.
\end{lemma}
\begin{proof} 
Recalling that by Lemma~\ref{L1} we have $\big\|S_Tu_0\big\|_{{\bf L}^{\infty}(\R)}\leq M$,
we shall rely on~\eqref{der-f-est-1}, \eqref{der-f-est-2} to show first that, for any $x<y$ such that~\eqref{sign-hyp}
holds, one has
\bel{f-est}
TV\big\{f'\circ u\ \big| \  [x,y]\big\}~\geq~ {\kappa_M^2\over 2}\cdot f'(u(x))\,,
\eeq
$\kappa_M\in (0,1)$ being the constant provided by Lemma~\ref{nc-flux-prop-1}.
We will consider only the case where 
\begin{equation}
\label{sign-hyp-2}
u(y)~<~0~<~u (x),
\end{equation}
the other case with $u (x)< 0 < u (y)$ being  entirely similar. We distinguish two sub-cases:
\begin{itemize}
\item[(i)] If there exists $z\in (x,y]$ with $u(z-)\in \big[0,\frac{u(x)}{2}\big]$,
by virtue of~\eqref{der-f-est-1} 
and since $f'$ is increasing on $[0, +\infty)$, we find
\begin{equation}
\label{TV-e-2}
\begin{aligned}
TV\big\{f'\circ u\ \big| \  [x,y]\big\}
~&\geq~\big|f'(u(x)\big)-f'(u(z-))\big|\\
\noalign{\smallskip}
~&=~f'(u(x))-f'(u(z-))
\\
\noalign{\smallskip}
~&\geq~f'(u(x))-f'(u(x/2))~\geq~\kappa_M\cdot f'(u(x)))\
\end{aligned}
\end{equation}
proving~\eqref{f-est}.

\item[(ii)] Otherwise, because of~\eqref{sign-hyp-2}, $S_T u_0$ must admit an admissible discontinuity 
located at some point $z\in [x,y]$, such that  the left state $u(z-)\in \big[{u(x)\over 2},\,u(x)\big]$ and 
the right state $u(z+)< 0$.
In the particular cases where $z=x$ or $z=y$, it must be $u(x)=u(z-)$
and $u(y)=u(z+)$, respectively.
Thus, one has
\begin{equation}
\label{TV-e-3}
TV\big\{f'\circ u\ \big|\   [x,y]\big\}~\geq~\big|f'(u(z-)\big)-f'(u(z+))\big|\,.
\end{equation}
Notice that
the Ole\v{\i}nik E-condition~\eqref{E-cond} implies 
\begin{equation}
\label{Ol-cond}
f'(u(z-))~\geq~f'(u(z+))\,.
\end{equation}
%
Since $f'$ is decreasing on $(-\infty,0]$, 
we then  obtain 
\begin{equation*}
\begin{aligned}
f(u(z-))-f(0)~
&=~\int^{u(z-)}_{0}f'(s)~ds
\\
&\geq~f'(u(z-))\cdot u(z-)~\geq~f'(u(z+))\cdot u(z-)\,,
\end{aligned}
\end{equation*}
which yields
\[
{f(u(z-))-f(0)\over u(z-)}~\geq~~f'(u(z+))\,.
\]
Thanks to~\eqref{der-f-est-2}, we thus deduce
\[
f'(u(z+))~\leq~{f(u(z-))-f(0)\over u(z-)}~\leq~\left(1-{\kappa_M\over 2}\right)\cdot f'(u(z-))\,,
\]
which, relying on~\eqref{der-f-est-1}, implies 
\begin{equation}
\label{TV-e-4}
\begin{aligned}
\big|f'(u(z+))-f'(u(z-))\big|
~&=f'(u(z-))-f'(u(z+))
\\
\noalign{\smallskip}
~&\geq~{\kappa_M\over {2-\kappa_M}}\cdot f'(u(z-))
%
\\
\noalign{\smallskip}
~&\geq~{\kappa_M\over 2}\cdot f'(u(x)/2)
\\
\noalign{\smallskip}
~&\geq~{\kappa_M^2\over {2}}\cdot f'(u(x))
\end{aligned}
\end{equation}
since $u(z-)\geq u(x)/2$, and because $f'$ is increasing
on $[0,+\infty)$.
Hence, \eqref{TV-e-3}, \eqref{TV-e-4} together yield~\eqref{f-est}.
\end{itemize}
Observing that 
\[
TV\big\{f'\circ u\ \big| \  [x,y]\big\}~\geq~\big|f'(u(y))- f'(u(x))\big|
\]
we derive from~\eqref{f-est} that
\begin{equation}
\label{TV-e-5}
\begin{aligned}
\bigg(1+{2\over \kappa_M^2}\bigg)\cdot TV\big\{f'\circ u\ \big| \  [x,y]\big\}~&~\geq~\big|f'(u(y))- f'(u(x))\big|+\big| f'(u(x))\big|
\\
\noalign{\smallskip}
~&~\geq~f'(u(y))\,.
\end{aligned}
\end{equation}
Therefore, \eqref{TV-e-5} implies
\[
TV\big\{f'\circ u\ \big| \  [x,y]\big\}~\geq~{\kappa_M^2\over\kappa_M^2+2}\cdot f'(u(y))\,,
\]
which, together with~\eqref{f-est},  yields (\ref{TV-e}) with
\[
\widetilde{\kappa}_M~\doteq~{\kappa_M^2\over\kappa_M^2+2}\,.
\]
\end{proof}

\begin{lemma}
\label{nc-flux-prop-2}
Let $f:\R\to\R$ be a differentiable map.
Given any $L,M>0$, for every $u,v\in {\bf L}^\infty(\R)$ with 
\begin{equation}
\label{hyp-inv-der-f-lip-est-nc1}
\begin{gathered}
\|u\|_{{\bf L}^{\infty}}~\leq~M, \ \|v\|_{{\bf L}^{\infty}}~\leq~M,\qquad\qquad u(x)\cdot v(x) \geq 0\qquad\forall~x\in\R\,,
\\
\noalign{\smallskip}
Supp(u)~\subset~ [-L,L], \  Supp(v)~\subset~ [-L,L],
\end{gathered}
\end{equation}
there holds
\begin{equation}
\label{inv-der-f-lip-est-nc1}
\big\|u-v\big\|_{{\bf L}^1(\R)}
~\leq~(1+2L)\cdot {\Delta}^{\!\strut -1}_{_{f,M}}\Big(\big\|f'\circ u-f'\circ v\big\|_{{\bf L}^1(\R)}\Big),
\end{equation}
where $\Delta_{_{f,M}}$ is the map defined in~\eqref{delta-funct-def}.
\end{lemma}
\begin{proof}
Observe that $s\mapsto \Delta_{_{f,M}}(s)$, $s\mapsto \frac{\Delta_{_{f,M}}(s)}{s}$ are strictly increasing maps
and that there holds
\begin{equation}
\label{inv-der-f-lip-est-2}
\hspace{1in}\Delta_{_{f,M}}\big(|u-v|\big)~\leq~|f'(u)-f'(v)|
\quad\qquad\forall~u,v\in\R\quad\text{s.t.}\quad |u|, |v|\leq M,
\quad u\cdot v\geq 0\,.
\end{equation}
Then, the estimate~\eqref{inv-der-f-lip-est-nc1} can be obtained with the same arguments of the proof of
Lemma~\ref{inv-der-f-lip-est} replacing $\widehat\Delta_{_{f,M}}$ with $\Delta_{_{f,M}}$
since, by assumption, $u(x)$ and $v(x)$ have the same sign for all $x\in\R$. 
\end{proof}
\medskip
The next lemma provides an estimate of the ${\bf L}^1$-distance between a given element $u\in S_T({\mathcal C}_{[L,M]})$
and its projection on the space of piecewise constant functions defined as follows. 
Fix $N\in\N$, letting $ l_{[L,M,T]}$ be the constant in~\eqref{LT-def}, set
\begin{equation}
\label{xnu-def}
x_{\nu}~\doteq~ -l_{[L,M,T]}+ {2\,  l_{[L,M,T]}\over N}\cdot\nu\,,
\qquad\quad
\nu\in \{0,1,..., N\}\,,
\end{equation}
and define (recalling from Remark~\ref{Rem:RightAndLeftLimits} that $u$ admits one-sided limits at each point $x$)
\begin{equation}
\label{proj-def}
\mathcal{P}^N(u)(x)~\doteq~ 
\begin{cases}
u(x_\nu^+)\ \quad\forall~x\in [x_\nu, x_{\nu+1})\quad&\text{if}\qquad x\in\big[\!-\!l_{[L,M,T]},\,l_{[L,M,T]}\big)\,,
\\
\noalign{\smallskip}
\ \  0\qquad&\text{otherwise.}
\end{cases}
\end{equation}
We shall express the ${\bf L}^1$-distance between $u\in S_T({\mathcal C}_{[L,M]})$ and $\mathcal{P}^N(u)$ in terms of $ TV \big\{f'\circ u\big\}$
which, in turn, admits an a-priori bound provided by Lemma~\ref{L2}.
\begin{lemma}
\label{nc-flux-prop-3}
Let $f:\R\to\R$ be a smooth map satisfying the assumption~ \eqref{hyp-NC2}.
Given any  $L,M,T>0$, for every $u\in S_T({\mathcal C}_{[L,M]})$,
and for any $N\in\N$, 
there holds
%
\begin{align}
\label{inv-der-f-lip-est-nc3}
\Big\|f'\circ u-f'\circ \mathcal{P}^N(u)\Big\|_{{\bf L}^1(\R)}
&\leq~\frac{2\,l_{[L,M,T]}\cdot TV \big\{f'\circ u\big\}}{N}\,,
\\
\noalign{\medskip}
\label{inv-der-f-lip-est-nc2}
\Big\|u-\mathcal{P}^N(u)\Big\|_{{\bf L}^1(\R)}
&\leq
4\,l_{[L,M,T]}\cdot \left(2+{TV \big\{f'\circ u\big\}\over \widetilde{\kappa}_M}\right)
\cdot {\Delta}^{\!\strut -1}_{_{f,M}}\Big(\frac{1}{N}\Big)
\end{align}
%
where $\widetilde{\kappa}_M$ is the constant provided by Lemma~\ref{k-est}.
\end{lemma}
\begin{proof}\hspace{1in}

\noindent
{\bf 1.} Observe first that by definition~\eqref{proj-def} there holds
\begin{equation}
\label{proj-est-1}
\Big|f'\big(u(x)\big)-f'\big(\mathcal{P}^N(u)(x)\big)\Big|~\leq~TV \big\{f'\circ u\ \big| \ [x_\nu, x_{\nu+1})\big\}\qquad\quad\forall~x\in [x_\nu, x_{\nu+1})\,, 
\end{equation}
for all $\nu\in \{0,1,..., N-1\}$. Hence, since  by~\eqref{linf-supp-nbounds-1}, \eqref{proj-def} one has $u= \mathcal{P}^N(u)=0$ on $\R\setminus [-l_{[L,M,T]}, \, l_{[L,M,T]}]$, 
we derive
\begin{equation}
\begin{aligned}
\Big\|f'\circ u-f'\circ \mathcal{P}^N(u)\Big\|_{{\bf L}^1(\R)}
&=~\int_{-l_{[L,M,T]}}^{-l_{[L,M,T]}} \Big|f'\big(u(x)\big)-f'\big(\mathcal{P}^N(u)(x)\big)\Big||dx
\\
\noalign{\smallskip}
&\leq~ {2\,  l_{[L,M,T]}\over N}\cdot \sum_{\nu=0}^{N-1} TV \big\{f'\circ u\ \big| \ [x_\nu, x_{\nu+1})\big\}
\\
\noalign{\smallskip}
&=~ {2\,  l_{[L,M,T]}\over N}\cdot TV \big\{f'\circ u\big\}\,,
\end{aligned}
\end{equation}
proving~\eqref{inv-der-f-lip-est-nc3}.
\medskip

{\bf 2.} Towards a proof of~\eqref{inv-der-f-lip-est-nc2}, we first show that, setting 
\begin{equation}
\label{ro-def-2}
\rho~\doteq~{\Delta}^{\!\strut -1}_{_{f,M}}\Big(\frac{1}{N}\Big),
\end{equation}
one has
\begin{equation}
\label{E2}
\qquad
\big|u(x)-
\mathcal{P}^N(u)(x)
\big|~\leq~ 2\, \rho \cdot \max\left\{2,\ \frac{N\cdot TV \big\{f'\circ u\ \big| \ [x_\nu, x_{\nu+1})\big\}}
{\widetilde\kappa_M}
\right\}\qquad\forall~x\in [x_\nu, x_{\nu+1}),
\end{equation}
for all $\nu\in \{0,1,..., N-1\}$.

Indeed, in the case where $u(x)$ and $\mathcal{P}^N(u)(x)=u(x_\nu)$, $x\in[x_\nu, x_{\nu+1})$, have the same sign, 
relying on~\eqref{inv-der-f-lip-est-2} and recalling that by~\eqref{linf-supp-nbounds-1}, \eqref{proj-def} we have $|u(x)|, |\mathcal{P}^N(u)(x)|\leq M$,
with the same arguments of the proof of 
Lemma~\ref{inv-der-f-lip-est}, replacing the definition of $\rho$ in~\eqref{ro-def} with~\eqref{ro-def-2} 
one obtains the estimate
\begin{equation}
\label{E3}
\big|u(x)-\mathcal{P}^N(u)(x)\big|~\leq~\rho\cdot\max\bigg\{1,\ N\cdot \Big|f'\big(u(x)\big)-f'\big(\mathcal{P}^N(u)(x)\big)\Big|
\bigg\}\qquad\forall~x\in\R\,.
\end{equation}
From~\eqref{E3} we immediately recover~\eqref{E2} because of~\eqref{proj-est-1} and since $\widetilde\kappa_M<1$.
\medskip

On the other hand, if $u(x)$ and $\mathcal{P}^N(u)(x)$, have different signs
and we assume that 
\begin{equation}
\frac{\big|f'(u(x))-f'(u(x_\nu))\big|}{\big|u(x)-u(x_\nu)\big|}~\geq~ 
\frac{\widetilde\kappa_M
}{2\,\rho\cdot N}\,,
\end{equation}
it follows 
\begin{equation}
\label{E4}
\begin{aligned}
\big|u(x)-u(x_\nu)\big|&
\leq~\frac{2\,\rho\cdot N\,\big|f'(u(x))-f'(u(x_\nu))\big|}{\widetilde\kappa_M}
%
\\
\noalign{\smallskip}
&\leq~\frac{2\,\rho\cdot N\cdot TV \big\{f'\circ u\ \big| \ [x_\nu, x_{\nu+1})\big\}}{\widetilde\kappa_M}
\end{aligned}
\end{equation}
which proves~\eqref{E2}.
\medskip

Therefore, it remains to consider the case where $u(x)$ and $\mathcal{P}^N(u)(x)=u(x_\nu)$, $x\in[x_\nu, x_{\nu+1})$, have different signs and there holds
\begin{equation}
\label{hyp-bound-osc-der-f}
\frac{\big|f'(u(x))-f'(u(x_\nu))\big|}{\big|u(x)-u(x_\nu)\big|}~<~
\frac{\widetilde\kappa_M}
{2\,\rho\cdot N}\,.
\end{equation}
Since $u(x)$, $u(x_\nu)$ have opposite signs,
one has
\begin{equation}
\label{diff-sign}
\big|u(x)-\mathcal{P}^N(u)(x)\big|~=~
\big|u(x)-u(x_\nu)\big|~=~\big|u(x)\big|+\big|u(x_\nu)\big|\,.
\end{equation}
Moreover, by Lemma~\ref{k-est} there holds
\begin{equation}
\label{tv-der-f-1}
TV\big\{f'\circ u\ \big|\  [x_\nu,\,x_{\nu+1})\big\}~\geq~ \widetilde{\kappa}_M\cdot \max \big\{f'(u(x)), f(u(x_\nu))\big\}\,.
\end{equation}
We now denote by $\pi(u)$, 
$u \in \R \setminus \left\{0\right\}$,
the unique point in $\R$ such that
\begin{equation}
 \label{pi-u-def}
 f'(u)=f'(\pi(u))
 \qquad \textrm{ and } \qquad
   \pi(u) \ne u
 \,,
\end{equation}
while we set  $\pi(0) \doteq 0$,
and we distinguish two sub-cases:
\begin{itemize}
\item[(i)] Assume that 
\begin{equation}
\label{hypot-i}
\max\big\{|u(x)|,\, |\pi(u(x))|\big\}~\geq~ \rho,\qquad\quad\max\big\{|u(x_\nu)|,\,|\pi(u(x_\nu))|\big\}~\geq~ \rho.
\end{equation}
Then,
recalling definition~\eqref{delta-funct-def} and that $f'(0)=0$,
and relying on the monotonicity of the map $s\to \frac{\Delta_{_{f,M}}(s)}{s}$,
we derive
\begin{equation}
\label{bound-i-1}
\begin{aligned}
\frac{f'(u(x))}{\max\big\{|u(x)|,\, |\pi(u(x))|\big\}}
&\geq~ 
\frac{{\Delta}_{_{f,M}}\big(\max\big\{|u(x)|,\, |\pi(u(x))|\big\}\big)}{\max\big\{|u(x)|,\, |\pi(u(x))|\big\}}~\geq~ \frac{{\Delta}_{_{f,M}}(\rho)}{\rho}\,,
\\
\noalign{\medskip}
\frac{f'(u(x_\nu))}{\max\big\{|u(x_\nu)|,\, |\pi(u(x_\nu))|\big\}}
&\geq~ 
\frac{{\Delta}_{_{f,M}}\big(\max\big\{|u(x_\nu)|,\, |\pi(u(x_\nu))|\big\}\big)}{\max\big\{|u(x_\nu)|,\, |\pi(u(x_\nu))|\big\}}~\geq~ \frac{{\Delta}_{_{f,M}}(\rho)}{\rho}\,.
\end{aligned}
\end{equation}
Hence, by virtue of~\eqref{ro-def-2}, \eqref{tv-der-f-1}, \eqref{bound-i-1},  we deduce
\begin{equation}
\label{bound-i-3}
\begin{aligned}
|u(x)|&\leq~\rho\cdot N\cdot f'(u(x))
~\leq~\frac{\rho\cdot N\cdot TV\big\{f'\circ u\ \big|\  [x_\nu,\,x_{\nu+1})\big\}}
{\widetilde{\kappa}_M}\,,
%
\\
\noalign{\medskip}
|u(x_\nu)|&\leq~\rho\cdot N\cdot  f'(u(x_\nu))
~\leq~\frac{\rho\cdot N\cdot TV\big\{f'\circ u\ \big|\  [x_\nu,\,x_{\nu+1})\big\}}
{\widetilde{\kappa}_M}\,,
\end{aligned}
\end{equation}
which, together with~\eqref{diff-sign}, yield~\eqref{E2}.

\item[(ii)] If~\eqref{hypot-i} is not verified and~\eqref{hyp-bound-osc-der-f} holds,
we claim that 
\begin{equation}
\label{hypot-ii}
|u(x)|~\leq~2\rho,\qquad\quad |u(x_\nu)|~\leq~ 2\rho,
\end{equation}
which, because of~\eqref{diff-sign} implies
\[
|u(x)-u(x_{\nu})|~\leq~4\,\rho\,,
\] 
proving~\eqref{E2}.
In fact, if~\eqref{hypot-i}, \eqref{hypot-ii} are not verified, then it must be
\begin{equation}
\begin{aligned}
&\min\Big\{\max\big\{|u(x)|,\, |\pi(u(x))|\big\}, \  \max\big\{|u(x_\nu)|,\,|\pi(u(x_\nu))|\big\}\Big\}~<~\rho\,,
\\
\noalign{\smallskip}
&\max\big\{|u(x)|,\,|u(x_\nu)|\big\}>2\rho\,.
\end{aligned}
\end{equation}
Let us assume that 
\begin{equation}
\label{hypot-ii-2}
\max\big\{|u(x)|,\, |\pi(u(x))|\big\}~<~\rho\,,\qquad\quad  |u(x_\nu)|~>~2\rho\,,
\end{equation}
(the other case $ \max\big\{|u(x_\nu)|,\,|\pi(u(x_\nu))|\big\}< \rho$, $|u(x)|> 2\rho$
being entirely similar).  In this case, by~\eqref{diff-sign} and since
$f'$ is decreasing on $(-\infty,0]$ and  increasing on~$[0,+\infty)$,
we have 
\begin{equation}
\label{diff-sign-2}
|u(x)-u(x_{\nu})|~\leq~2|u(x_{\nu})|\,,
\qquad\quad
f'(u(x_\nu))~>~f'\left(\frac{u(x_{\nu})}{2}\right)~\geq~f'(u(x))\,.
\end{equation}
Thus, relying on~\eqref{ro-def-2}, \eqref{hyp-bound-osc-der-f}, \eqref{diff-sign-2}, we find
\begin{equation}
\label{bound-ii-2}
\begin{aligned}
\frac{{\Delta}_{_{f,M}}(\rho)}
{\rho} &=~\frac{1}{N\cdot\rho}
%
\\
\noalign{\smallskip}
&>~ \frac{2\cdot \big|f'(u(x))-f'(u(x_\nu))\big|}{\widetilde\kappa_M\cdot\big|u(x)-u(x_\nu)\big|}
~=~ \frac{2\cdot\big(f'(u(x_\nu))-f'(u(x))\big)}{\widetilde\kappa_M\cdot\big|u(x)-u(x_\nu)\big|}
\\
\noalign{\smallskip}
&\geq~ \frac{2\cdot\big(f'(u(x_\nu))-f'(u(x_\nu)/2)\big)}{\widetilde\kappa_M\cdot\big|u(x)-u(x_\nu)\big|}
~\geq~ \frac{f'(u(x_\nu))}{|u(x_{\nu})|}
\\
\noalign{\smallskip}
&\geq~ \frac{{\Delta}_{_{f,M}}(|u(x_{\nu})|)}
{|u(x_{\nu})|}\,.
\end{aligned}
\end{equation}
The increasing property of $s\to \frac{\Delta_{_{f,M}}(s)}{s}$ 
together with~\eqref{bound-ii-2} then implies $|u(x_{\nu})|\leq \rho$ which yields a contradiction
with~\eqref{hypot-ii-2}.  Thus, the bounds in~\eqref{hypot-ii} hold and the proof of~\eqref{E2} is complete.
\end{itemize}
\medskip

{\bf 3.} Since  by~\eqref{linf-supp-nbounds-1}, \eqref{proj-def} one has $u= \mathcal{P}^N(u)=0$ on $\R\setminus [-l_{[L,M,T]}, \, l_{[L,M,T]}]$, relying on~\eqref{E2} we find
\begin{equation}
\begin{aligned}
\Big\|u-\mathcal{P}^N(u)\Big\|_{{\bf L}^1(\R)}
&\leq~ \sum_{\nu=0}^{N-1} \Big\|u-\mathcal{P}^N(u)\Big\|_{{\bf L}^1([x_\nu,\,x_{\nu+1}])}
\\
\noalign{\smallskip}
&\leq~ \frac{2\,  l_{[L,M,T]}}{N}\cdot \sum_{\nu=0}^{N-1} \sup_{\ x\in[x_\nu,\,x_{\nu+1})} 
\big|u(x)-\mathcal{P}^N(u)(x)\big|
\\
\noalign{\smallskip}
&\leq~ 8\,l_{[L,M,T]}\cdot \rho~+\frac{4\,l_{[L,M,T]}\cdot \rho}{\widetilde\kappa_M}\cdot \sum_{\nu=0}^{N-1} TV \big\{f'\circ u\ \big| \ [x_\nu, x_{\nu+1})\big\},
\end{aligned}
\end{equation}
which yields~\eqref{inv-der-f-lip-est-nc2}.
\end{proof}
We are now ready to provide the:

\medskip

{\bf Proof of upper bound~\eqref{uest-nc} of Theorem~\ref{ThmLUBNC}}\\

By virtue of Lemma~\ref{entr-f'-sol-set}, given any 
\begin{equation}
\label{eps-hyp-2}
0~<~\varepsilon~<~ \big(2+4\, l_{[L,M,T]}\big) \cdot {\Delta}^{\strut -1}_{_{f,M}}\Big( \frac{\Gamma_{\!1}^+}{144}\Big),
\end{equation}
with $l_{[L,M,T]}$ as in~\eqref{LT-def}, and setting
\begin{equation}
\label{epsprime-def-2}
\varepsilon'~\doteq~ \frac{1}{2}\cdot  {\Delta}_{_{f,M}}\bigg(\frac{\varepsilon}{2+4\, l_{[L,M,T]}}\bigg)\,,
\end{equation}
there exists a set of piecewise constant functions
\begin{equation}
\label{G-set-def}
\mathcal{G}~\doteq~\{g_1,\dots,g_p\},
\end{equation}
with 
\begin{equation}
\label{cover-bound-nc-1}
p~\leq~\Bigg\lfloor2^{\Big(\frac{\Gamma_{\!1}^+}{\small\strut 2\,\varepsilon'}\Big)}\Bigg\rfloor+1\,,
\end{equation}
that enjoy the following properties:
\begin{itemize}
\item[(i)]  For any $i=1,\dots, p$, one has 
\[
Supp(g_i)~\subseteq~\big[\!- l_{[L,M,T]},\, l_{[L,M,T]}\big],
\qquad\quad Im(g_i) \subseteq \big[0, f'_M\big]\,,
\]
and
\begin{equation*}
\hspace{-20pt}
g_i(x)~=~g_i(x_\nu)
\qquad\quad\forall~x\in[x_\nu,\,x_{\nu+1})\,,
\qquad \nu\in\{0,1,\dots,N\!-\!1\},
\end{equation*}
with
\begin{equation}
\label{hyp-N-1}
x_{\nu}~\doteq~ -l_{[L,M,T]}+ {2\,  l_{[L,M,T]}\over N}\cdot\nu\,,
\qquad
\nu\in \{0,1,..., N\}\,,
\qquad\quad
N\geq   \bigg\lfloor\frac{8\, l_{[L,M,T]}\cdot V_{[L,M,T]}}{\varepsilon'}\bigg\rfloor,
\end{equation}
where  $f'_M$, $l_{[L,M,T]}$,  $V_{[L,M,T]}$ are  the constants defined
in~~\eqref{der-flux-bound}, \eqref{LT-def}, \eqref{VLT-def}, respectively.

\item[(ii)]
\begin{equation}
\label{prop-ii-gi}
\mathcal{L}_{[L,M,T]}~\subset~\bigcup_{i=1}^p B\big(g_i,\,\ve'\,\big)\,.
\end{equation}
\end{itemize}
For every $g_i$, $i=1,\dots, p$, and in connection with any $N$-tuple $\iota=(\iota_0,\dots,\iota_{N-1})\in\{-1,1\}^N$,
we now define a piecewise constant map $\mathcal{T}^N_\iota(g_i)$ as follows.
Let $f'_{-1}, f'_1$ denote the restrictions of $f'$ to the semilines $(-\infty, 0]$ and $[0,+\infty)$,
respectively. Then,  set
\begin{equation}
\label{T-iota-def}
\mathcal{T}^N_\iota(g_i)(x)~\doteq~ 
\begin{cases}
\big( f'_{\iota_\nu}\big)^{-1}\big(g_i(x_\nu)\big)\ \quad\forall~x\in [x_\nu, x_{\nu+1})\quad&\text{if}\qquad x\in\big[\!-\!l_{[L,M,T]},\,l_{[L,M,T]}\big)\,,
\\
\noalign{\medskip}
\ \  0\qquad&\text{otherwise.}
\end{cases}
\end{equation}
Next, given any $u\in S_T({\mathcal C}_{[L,M]})$, by~\eqref{derf-sol-incl1} and \eqref{prop-ii-gi}
let $g_i$ be a map satisfying property (i) such that
\begin{equation}
\label{gi-est-1}
\big\|f'\circ u-g_i\big\|_{{\bf L}^1(\R)}~<~\ve'\,.
\end{equation}
Observe that, applying Lemma~\ref{L2} and Lemma~\ref{nc-flux-prop-3},
and choosing
\begin{equation}
\label{N-est-1}
N~\geq~\left\lfloor\frac{2\,l_{[L,M,T]}\cdot C_1\,\Big(1+\frac{L}{T}\Big)}{\ve'}\right\rfloor+1\,,
\end{equation}
we find
\begin{equation}
\label{inv-der-f-lip-est-nc4}
\Big\|f'\circ u-f'\circ \mathcal{P}^N(u)\Big\|_{{\bf L}^1(\R)}
~\leq~\frac{2\,l_{[L,M,T]}\cdot  C_1\,\Big(1+\frac{L}{T}\Big)}{N}
~\leq~ \ve'\,.
\end{equation}
Hence, \eqref{gi-est-1}, \eqref{inv-der-f-lip-est-nc4} imply that, for any 
\begin{equation}
\label{N-est-2}
N~\geq~\max\left\{\Bigg\lfloor\frac{8\, l_{[L,M,T]}\cdot V_{[L,M,T]}}{\varepsilon'}\Bigg\rfloor,\ \left\lfloor\frac{4\,l_{[L,M,T]}\cdot C_1\,\Big(1+\frac{L}{T}\Big)}{\ve'}\right\rfloor\right\}\,,
\end{equation}
and $\ve'\leq 2\,l_{[L,M,T]}\cdot C_1$,
one has
\begin{equation}
\label{gi-est-2}
\Big\|f'\circ \mathcal{P}^N(u) -g_i\Big\|_{{\bf L}^1(\R)}~<~2\,\ve'\,.
\end{equation}
Let $\overline\iota\in\{-1,1\}^N$ be the $N$-tuple defined by
\begin{equation}
\label{bar-iota-def}
\overline\iota_\nu~=~\mathrm{sign}\big(u(x_\nu)\big)\qquad\quad
\nu\in\{0,1,\dots,N\}\,.
\end{equation}
Notice that, by definitions~\eqref{proj-def}, \eqref{T-iota-def},
by Lemma~\ref{L1}
and since $f'(0)=0$
and $g_i$ satisfies the property (i),
one has
\begin{equation}
\label{hyp-inv-P-T}
\begin{gathered}
\big\|\mathcal{P}^N(u)\big\|_{{\bf L}^{\infty}}~\leq~ M, \qquad\quad \big\| \mathcal{T}^N_{\overline \iota}(g_i)\big\|_{{\bf L}^{\infty}}~\leq~ M,
\\
\noalign{\medskip}
\qquad\qquad\qquad
\mathrm{sign}\big( \mathcal{P}^N(u)(x)\big)~=~ \mathrm{sign}\big(\mathcal{T}^N_{\overline \iota}(g_i)(x) \big)
\qquad\forall~x\in\R\,,
\\
\noalign{\medskip}
\qquad
Supp\big(\mathcal{P}^N(u)\big)~\subset~ \big[-l_{[L,M,T]},\, l_{[L,M,T]}\big], \qquad\quad  
Supp\big(\mathcal{T}^N_{\overline \iota}(g_i)\big)~\subset~ \big[-l_{[L,M,T]},\,l_{[L,M,T]}\big].
\end{gathered}
\end{equation}
Therefore, observing that $f'\circ \mathcal{T}^N_{\overline \iota}(g_i)= g_i$,
applying Lemma~\ref{L2}, Lemma~\ref{nc-flux-prop-2}, Lemma~\ref{nc-flux-prop-3} and relying on~\eqref{epsprime-def-2}, \eqref{gi-est-2},
we find that,  for all
\begin{equation}
\label{N-est-3}
N\!\geq\! \max\left\{\!
\left\lfloor \frac{1}{{\Delta}_{_{f,M}}\Big(\frac{\widetilde\kappa_M\, \varepsilon}{8\,l_{[L,M,T]}\cdot \big(2\widetilde\kappa_M+C_1(1+\frac{L}{T})\big)}\Big)}\right\rfloor\!\!,\!
\left\lfloor\frac{16\, l_{[L,M,T]}\cdot V_{[L,M,T]}}{ {\Delta}_{_{f,M}}\left(\frac{\varepsilon}{2+4\, l_{[L,M,T]}}\right)}
\right\rfloor\!\!,\!
\left\lfloor\frac{8\,l_{[L,M,T]}\cdot C_1\big(1\!+\!\frac{L}{T}\big)}{ {\Delta}_{_{f,M}}\left(\frac{\varepsilon}{2+4\, l_{[L,M,T]}}\right)}\right\rfloor\!\right\}\!,
\end{equation}
there holds
\begin{equation}
\label{gi-est-3}
\begin{aligned}
\Big\|\mathcal{P}^N(u)-\mathcal{T}^N_{\overline \iota}(g_i)\Big\|_{{\bf L}^1(\R)}
&\leq~\big(1+2l_{[L,M,T]}\big)\cdot {\Delta}^{\!\strut -1}_{_{f,M}}(2\ve')~\leq~\ve/2\,,
\\
\noalign{\medskip}
\Big\|u-\mathcal{P}^N(u)\Big\|_{{\bf L}^1(\R)}
&\leq \frac{4 \, l_{[L,M,T]}}{\widetilde\kappa_M}
\bigg(2\widetilde\kappa_M+C_1\Big(1+\frac{L}{T}\Big)\bigg)
\cdot {\Delta}^{\!\strut -1}_{_{f,M}}\Big(\frac{1}{N}\Big)\leq~\ve/2\,.
\end{aligned}
\end{equation}
Hence, by~\eqref{gi-est-3},  for any given  $u\in S_T({\mathcal C}_{[L,M]})$ 
and for every $N$ satisfying~\eqref{N-est-3},
we can find
an element $g_i$ of the set $\mathcal{G}$ in~\eqref{G-set-def} and an $N$-tuple 
$\overline\iota\in\{-1,1\}^N$ such that
\begin{equation*}
\Big\|u-\mathcal{T}^N_{\overline \iota}(g_i)\Big\|_{{\bf L}^1(\R)}~\leq~\ve\,,
\end{equation*}
showing that
\begin{equation}
\bigcup_{\ \iota\in\{-1,1\}^N\  } \bigcup_{i=1}^p B\big(\mathcal{T}^N_{\iota}(g_i),\,\ve\big)
\end{equation}
provides an $\ve$-cover of $S_T({\mathcal C}_{[L,M]})$ in ${\bf L}^1$ of cardnality $p\cdot 2^N$ 
By virtue of~\eqref{cover-bound-nc-1}, \eqref{N-est-3}, 
for $\ve>0$ sufficiently small
one has
%
%
\begin{equation}
\label{card-est-3}
p\cdot 2^N~\leq~
2^{\Big(\frac{\Gamma^+}{
{\Delta}_{_{f,M}}(\varepsilon/\gamma^+)}\Big)}
\end{equation}
with  
\begin{equation}
\begin{aligned}
\Gamma^+&\doteq~2\Gamma_1^+ + 
\max\Big\{
2,\, 32\, l_{[L,M,T]}\!\cdot\! V_{[L,M,T]},\, 16\,l_{[L,M,T]}\!\cdot\! C_1\,\big(1+{L}/{T}\big)
\Big\},
\\
\noalign{\medskip}
\gamma^+&\doteq\max\Bigg\{
\frac{8 \, l_{[L,M,T]}}{\widetilde\kappa_M}
\bigg(2\widetilde\kappa_M+C_1\Big(1+\frac{L}{T}\Big)\bigg),
\ 2+4\, l_{[L,M,T]}\Bigg\}\,.
\end{aligned}
\end{equation}
Recalling definitions~\eqref{G1+-def}, \eqref{LT-def}, \eqref{VLT-def} we deduce that
there exists some constant $c>1$ such that
\begin{equation}
\label{card-est-4}
\Gamma^+\leq \eta
\,,\quad \gamma^+\leq \eta
\,,\qquad\quad
\eta\doteq c \bigg(1+L+T+\frac{L^2}{T}\bigg)\,.
\end{equation}
Thus, relying on~\eqref{der-f-est-3}, \eqref{card-est-3}, \eqref{card-est-4},
it follows that
there holds
\begin{equation}
\mathcal{N}_{\varepsilon}\Big(S_T({\mathcal C}_{[L,M]}) \ | \ {\bf L}^{1}(\R)\Big) \leq2^{\big(\frac{\eta\strut}{
{\Delta}_{_{f,M}}(\varepsilon/\eta)}\big)}\leq 2^{\left(\frac{\Gamma_2^+}{\ve^m}\right)}
\end{equation}
with
\begin{equation}
\label{card-est-5}
\Gamma_2^+\doteq  \beta_M\cdot \eta^{m+1}\,.
\end{equation}
Taking the base-2 logarithm in~\eqref{card-est-5} we then derive the estimate~\eqref{uest-nc}.
\qed

\section{Lower compactness estimates} 
\label{sec:lce}
In this section we derive
lower bounds on the $\ve$-entropy in $\L^1$ of
 $S_T(\mathcal{C}_{[L,M]})$ for the class of initial data $\mathcal{C}_{[L,M]}$ in~\eqref{DefCLM},
 when the flux function $f$ satisfies the assumption:
\begin{itemize}
\item[{\bf (A)}] $f:\R\to\R$ is a twice continuously differentiable map such that
\[
f'(0)~=~f''(0)~=~0,\qquad\qquad f''(x)~\neq~0\qquad\forall~x\in\R\backslash \{0\}\,,
\]
\end{itemize}
 which is fulfilled by fluxes satisfying~\eqref{EqZeroSpeed} and either of the assumptions {\bf (C)} or {\bf(NC)}
 stated in the Introduction. Notice that {\bf (A)} in particular implies that $f''$ does not change sign on the
 two semilines $(-\infty, 0)$ and $(0,\infty)$.

 Following the same approach introduced in~\cite{AON1}, we shall derive a proof of~\eqref{lest-c}, \eqref{lest-nc} relying
 on a controllability results for BV functions with one-side bounds on their spatial distributional derivative. 
  Namely, given any $L,h,T>0$, 
  setting
 \begin{equation}
 \label{bh+--def}
b^+_h~\doteq~\frac{1}{2T \cdot\displaystyle{\max_{z\in [0, h]}} |f''(z)|},\qquad\qquad b^-_h ~\doteq~ \frac{1}{2T \cdot\displaystyle{\max_{z\in [-h, 0]}} |f''(z)|}\,,
\end{equation}
consider the sets
  \begin{equation}
  \label{A+--def}
  \begin{aligned}
  \mathcal{A}^+_{[L,\,h]}
  ~\!\doteq\!~
  \begin{cases}
  \!\Big\{v\in {\mathcal C}_{[L/2,\,h]}\cap BV(\R)~\big|~v(x)\geq 0\ \ \forall~x\in\R,\quad
Dv ~\leq b_h^+ \Big\}\quad
&\text{if}\qquad f''(h)>0,
\\
\noalign{\bigskip}
 \!\Big\{v\in {\mathcal C}_{[L/2,\,h]}\cap BV(\R)~\big|~v(x)\geq 0\ \ \forall~x\in\R,\quad
Dv ~\geq -b_h^+ \Big\}\quad
&\text{if}\qquad f''(h)<0,
\end{cases}
\\
\noalign{\medskip}
\mathcal{A}^-_{[L,\,h]}
 ~ \!\doteq\!~
  \begin{cases}
  \!\Big\{v\in {\mathcal C}_{[L/2,\,h]}\cap BV(\R)~\big|~v(x)\leq 0\ \ \forall~x\in\R,\quad
Dv ~\leq b_h^- \Big\}\quad
&\text{if}\qquad f''(-h)>0,
\\
\noalign{\bigskip}
 \!\Big\{v\in {\mathcal C}_{[L/2,\,h]}\cap BV(\R)~\big|~v(x)\leq 0\ \ \forall~x\in\R,\quad
Dv ~\geq -b_h^- \Big\}\quad
&\text{if}\qquad f''(-h)<0\,.
\end{cases}
  \end{aligned}
  \end{equation}
  %
   Here and throughout the following,  the inequalities of the form $Du\geq b$
 for a function $u\in BV(\R)$,
must  be understood in the sense of measures, i.e. the Radon measure $D u$
satisfies $D u (J)\geq b\cdot|J|$
for every Borel set $J\subset\R$, $|J|$ being the Lebesgue measure of $J$.
We will show that  any element of $\mathcal{A}^\pm_{[L,h]}$ can be obtained as the
value at time $T$ of a
solution of~\eqref{EqCL} with initial data in the set $\mathcal{C}_{[L,h]}$ in~\eqref{DefCLM}.
  To this end, the following lemma provides a-priori bounds on the spatial distributional derivative
of an entropy solutions of~\eqref{EqCL}.
\begin{lemma}\label{reg} 
Let $f:\R\to\R$ be a map satisfying the assumption {\bf (A)} and,
given  $L, h, T>0$, let  $u_0\in\mathcal{C}_{[L,h]}\cap BV(\R)$  be any function
satisfying either of the conditions:
%
\begin{equation}
\label{sign+u0-hyp}
u_0(x)~\geq~ 0\quad\quad\forall~x\in\R\,,
\qquad\qquad
\mathrm{sign}\big(f''(u_0(h))\big)\cdot Du_0\geq - b_h^+\,,
\qquad
\end{equation}
%
\begin{equation}
\label{sign-u0-hyp}
u_0(x)~\leq~ 0\quad\quad\forall~x\in\R\,,
\qquad\qquad
\mathrm{sign}\big(f''(u_0(-h))\big)\cdot Du_0\geq - b_h^-\,,
\qquad
\end{equation}
where $b_h^\pm$ are the constants defined in~\eqref{bh+--def}.
Then,  for every $t\in (0,T]$,
the entropy solution $u(t,\cdot)\doteq S_t u_0$ is continuous on $\R$ and
 one has
\begin{equation}
\label{cd2}
\begin{aligned}
\mathrm{sign}\big(f''(u_0(h))\big)\cdot D u(t,\cdot)~\geq~
- 2 b_h^+
\qquad \ &\text{if}\qquad\ \ \eqref{sign+u0-hyp}\quad \text{holds}\,,
\\
\noalign{\bigskip}
\mathrm{sign}\big(f''(u_0(-h))\big)\cdot D u(t,\cdot)~\geq~
- 2 b_h^-
\qquad \ &\text{if}\qquad\ \ \eqref{sign-u0-hyp}\quad \text{holds}\,.
\end{aligned}
\end{equation}
\end{lemma}
\begin{proof} 
We shall consider only the case where $u_0$ satisfies condition~\eqref{sign+u0-hyp} and $f''(u_0(h))\geq 0$.
The cases where $f''(u_0(h))\leq 0$ or where condition~\eqref{sign-u0-hyp} holds can be treated in an entirely similar way.

{\bf 1.} Assume that~\eqref{sign+u0-hyp} holds and that $f'$ is increasing on $[0, +\infty)$.
Observe first that, by Lemma~\ref{L1}, we have 
$u(t,\cdot)\in \mathcal{C}_{[l_{[L,h,t]},h]}\cap BV(\R)$,
$u(t,x)\geq 0$ 
for any $x\in\R$, $t>0$, and that~\eqref{cd2} in particular implies
\begin{equation}
\label{cd2-a}
u(t,x+) - u(t,x-)~=~D u(t,\cdot)\big(\{x\}\big)~\geq~0
\qquad\quad\forall~x\in\R\,.
\end{equation}
On the other hand, by the Ole\v{\i}nik E-condition~\cite{Oleinik} we have
\begin{equation*}
f'\big(u(t,x-)\big)~\geq ~f'\big(u(t,x+)\big)
\qquad\quad\forall~x\in\R\,, \ t>0\,,
\end{equation*}
which, in turn, by the  monotonicity of $f'$ on $[0, +\infty)$, implies 
\begin{equation}
\label{cd2-b}
u(t,x-)~\geq~u(t,x+)
\qquad\quad\forall~x\in\R\,, \ t >0\,,
\end{equation}
Then, \eqref{cd2-a}-\eqref{cd2-b} together yield
\begin{equation}
\label{cd2-c}
u(t,x-)~=~u(t,x+)
\qquad\quad\forall~x\in\R\,, \ t\in (0,T]\,,
\end{equation}
proving the continuity of $S_t u_0$ at any $x\in\R$ and for any $t\in(0,T]$.
Therefore, to complete the proof of the Lemma we only have to show that, if the initial data $u_0$ satisfies 
the assumption~\eqref{sign+u0-hyp}, then the corresponding entropy solution satisfies
the inequality in~\eqref{cd2} which, in this case, is equivalent to
\begin{equation}
\label{cd2-d}
u(t,x_2+)-u(t,x_1-)~\geq~ 
-\frac{x_2-x_1}{T \cdot\displaystyle{\max_{z\in [0, h]}} |f''(z)|}
\qquad\quad\forall~x_1<x_2\,.
\end{equation}
Clearly, it will be sufficient to prove that the inequality in~\eqref{cd2-d} holds for any pair of continuity
points $x_1<x_2$ of $u(t,\cdot)$ such that 
\bel{as1}
u(t,x_2)-u(t,x_1)~<~0\,.
\eeq
\medskip

{\bf 2.} Because of~\eqref{sign+u0-hyp}, and since we are assuming that $f''(u_0(h))\geq 0$,
the initial data $u_0$ satisfies the inequality
\begin{equation}
\label{sign+u0-hyp-b}
u_0(z_2+)-u_0(z_1-)~\geq~ 
-\frac{z_2-z_1}{2T \cdot\displaystyle{\max_{z\in [0,h]}} |f''(z)|}
\qquad\quad\forall~z_1<z_2\,.
\end{equation}
Notice that, since $u(t,\cdot)$ takes values in the semiline $[0,+\infty)$ for all $t>0$, 
we may always view $u(t,x)$ as the entropy solution of a conservation law with convex flux.
In fact, if $f$ satisfies the  assumption  {\bf(NC)}, $u(t,x)$ turns out to be the entropy solution
of 
\begin{equation} \label{EqCLNC}
u_{t} + \widetilde f(u)_{x}~=~0,
\end{equation}
with
\begin{equation*}
\widetilde f(u)~\doteq~
\begin{cases}
f(u)\quad &\text{if}\qquad u\geq 0\,,
\\
\noalign{\medskip}
2f(0)-f(u)\quad &\text{if}\qquad u\leq 0\,,
\end{cases}
\end{equation*}
where $\widetilde f$ is a twice continuously differentiable convex map.
Therefore, we may employ the theory of generalized characteristics of Dafermos~\cite{dafermos:gc, Dafermos:Book}
and, for every given point $x$ of continuity of $u(t,\cdot)$, we may trace a unique
backward characteristic starting at $(t,x)$ that is a genuine characteristic.

Then, fix $t\in (0,T]$ and consider two continuity points $x_1<x_2$ of $u(t,\cdot)$ such that~\eqref{as1} holds.
Let $\xi_i(\cdot)$ be the unique backward characteristics emanating from $(t,x_i)$ for $i=1,2$. Since the solution $u(t,\cdot)$ is constant along genuine characteristics, we have 
\begin{equation}
\label{char-eq-1}
x_i~=~y_i+t\cdot f'(u_0(y_i))\qquad\mathrm{with}\qquad y_i~=~\xi_i(0)\,,
\end{equation}
and 
\begin{equation}
\label{char-eq-2}
u(t,x_i)~=~u_0(y_i)\qquad\mathrm{for}~~~i=1,2\,.
\end{equation}
Notice that  \eqref{as1}, \eqref{char-eq-1}, \eqref{char-eq-2}
and the   monotonicity of $f'$ on $[0, +\infty)$, together imply 
\begin{equation}
\label{char-neq-1}
\begin{gathered}
y_2-y_1  = x_2-x_1 - t \cdot \big( f'(u(t, x_2))- f'(u(t, x_1))\big)~>~0\,,
\\
\noalign{\smallskip}
f'(u_0(y_2))~<~ f'(u_0(y_1))\,.
\end{gathered}
\end{equation}
Thus, relying on  \eqref{sign+u0-hyp-b}, \eqref{char-eq-1}, \eqref{char-eq-2}, \eqref{char-neq-1}
we find
\bel{cd3}
u(t,x_2)-u(t,x_1)~=~u_0(y_2)-u_0(y_1)~\geq~-\frac{y_2-y_1}{2T \cdot\displaystyle{\max_{z\in[0, h]}} |f''(z)|}
\eeq
and
\begin{eqnarray}
\label{char-eq-3}
x_2-x_1&=&y_2-y_1+t\cdot \big( f'(u_0(y_2))- f'(u_0(y_1))\big)\cr\cr
&\geq &y_2-y_1+ t\cdot \Big(\,\displaystyle{\max_{ z\in [0, h]}} |f''(z)|\Big)
\cdot \big(u_0(y_2)-u_0(y_1)\big)\cr\cr
&\geq&y_2-y_1-t\cdot \Big(\,\max_{z\in [0,h]}~|f''(z)|\Big)\cdot {y_2-y_1\over 2T\cdot \displaystyle{\max_{z\in [0,h]}}~f''(z)}\cr\cr
&\geq &{y_2-y_1\over 2}\,.
\end{eqnarray}
Combining~\eqref{cd3}, \eqref{char-eq-3}, we obtain
\[
u(t,x_2)-u(t,x_1)~\geq~- \frac{x_2-x_1}{T \cdot\displaystyle{\max_{z\in [0,h]}} |f''(z)|}
\]
completing the proof of (\ref{cd2-d})  for any pair of continuity
points $x_1<x_2$ of $u(t,\cdot)$
and thus concluding the proof of the Lemma.
\end{proof}
Relying on Lemma \ref{reg}, we obtain the following controllability result.
\begin{lemma}\label{inc2}
Let $f:\R\to\R$ be a map satisfying the assumption {\bf (A)} and,
given  $L, h, T>0$, let  $\mathcal{C}_{[L,h]}$, $\mathcal{A}^\pm_{[L,h]}$ be the sets defined in~\eqref{DefCLM}, \eqref{A+--def}, respectively. Then, there holds
\bel{inc}
\mathcal{A}^+_{[L,\,h]}\,\bigcup\,\mathcal{A}^-_{[L,\,h]}~\subseteq~S_T(\mathcal{C}_{[L,h]})
\eeq
for all 
$h>0$ 
such that 
\begin{equation}
\label{h-controll-bound}
f'_h\doteq \max_{|z|\leq h}~|f'(z)|~\leq~{L\over 2T}\,.
\end{equation}
\end{lemma}
\begin{proof}
We will only show that, for $h$ satisfying~\eqref{h-controll-bound},
assuming $f''(h)>0$
one has  
\begin{equation}
\label{A+incl}
\mathcal{A}^+_{[L,\,h]}~\subseteq~ S_T(\mathcal{C}_{[L,h]}).
\end{equation}
The proof of~\eqref{A+incl}
when $f''(h)<0$ and the
proof of 
$\mathcal{A}^-_{[L,\,h]}\subseteq S_T(\mathcal{C}_{[L,h]})$ are entirely similar.
Then, given an arbitrary function 
\begin{equation}
\label{v}
v~\in~ \mathcal{A}^+_{[L,\,h]},
\end{equation}
we will determine an element $u_0\in \mathcal{C}_{[L,h]}$
such that
\begin{equation}
S_T u_0~=~v\,,
\end{equation}
thus proving~\eqref{A+incl}.
The function $u_0$ will be obtained by 
an entropy admissible solution of~\eqref{EqCL}
backward constructed in time, which
starts at time $T$ with the value $v$. Namely,  set
\begin{equation}
\label{w0-def}
w_0(x)~\doteq~v(-x)\qquad\forall~x\in\R\,,
\end{equation}
and consider the entropy weak solution $w(t,x)\doteq S_t w_0$
of~\eqref{EqCL} with initial data $w_0$. 
Notice that, letting $l_{[L/2,h,t]}$ be the constant defined  in~\eqref{LT-def},
because of~\eqref{h-controll-bound} there holds
\begin{equation}
\label{supp-size-bound}
l_{[L/2,h,t]}~ = ~ L/2 + t\cdot f'_h~\leq~ L\qquad \forall~t\in [0,T]\,.
\end{equation}
Moreover, observe that, by~\eqref{A+--def}, \eqref{v}, \eqref{w0-def}, and since we are assuming that $f''(h)>0$,
we have
\begin{gather}
\label{w0-cond1}
w_0~\in~\mathcal{C}_{[L/2,h]}\cap BV(\R)\,,
\\
\noalign{\medskip}
v(x)~\geq~ 0\quad \forall~x\in\R,\qquad
Dw_0~=-~Dv  ~\geq~ -b_h^+\,.
\end{gather}
Therefore, by virtue of Lemma~\ref{L1} we find
\begin{equation}
\label{linf-supp-size-bound-2}
\big\|w(t,\cdot)\big\|_{{\bf L}^{\infty}(\R)}~\leq~ h\,,\qquad\quad
Supp(w(t,\cdot))~\subseteq~ \big[\!-\!L, L\big]
\qquad\forall~t\in [0,T]\,,
\end{equation}
and  invoking Lemma~\ref{reg} we deduce that $w(t,\cdot)$
is a continuous map on $\R$ for all $t\in (0,T]$. 
Next, observe that the map $u$ defined by
\begin{equation}
\label{u-def}
u(t,x)~\doteq~ w(T-t, -x)\qquad\quad (t,x)\in [0,T]\times\R\,,
\end{equation}
provides a weak distributional solution of~\eqref{EqCL} which is entropy admissible since it is 
continuous with respect to the space variable $x$ at any time $t < T$.
On the other hand, by~\eqref{w0-def}, \eqref{linf-supp-size-bound-2}, \eqref{u-def}, we have
\begin{equation}
\label{u-0t-prop}
u_0~\doteq~ u(0,\cdot)\in \mathcal{C}_{[L,h]}\,,
\qquad\qquad
S_T u_0~=~ u(T,\cdot)~=~v\,,
\end{equation}
which completes the proof of the Lemma.
\end{proof}
\medskip 
The next lemma shows that, for fluxes with polynomial degeneracy at zero, the constants $b_h^\pm$ in~\eqref{bh+--def}
are of order $\approx \frac{1}{T\cdot s^{m-1}}$.
\begin{lemma}
\label{up-bounds-second-der-f} 
Assume that $f:\R\to\R$ is a  function satisfying condition~\eqref{EqZeroSpeed} and either 
of \eqref{hyp-NC}
or \eqref{hyp-C-pol} conditions.
Then, there exist constants $\overline\alpha, \overline\sigma>0$ such that
\begin{equation}
\label{second-der-bound-2}
\max\Big\{\displaystyle{\max_{z\in [0, s]}} |f''(z)|,\,\displaystyle{\max_{z\in [-s, 0]}} |f''(z)|\Big\}
~\leq~\overline\alpha \cdot s^{m-1}
\qquad\quad\forall~s\in [0,\overline\sigma]\,.
\end{equation}
\end{lemma}
\begin{proof}
By writing the 
Taylor expansion of $f''$ at zero we find
\begin{equation}
\label{der-3-f-taylor}
f''(u)~=~ 
u^{m-1}\cdot \Bigg( \frac{f^{(m+1)}(0)}{(m-1)!}+o(1)\Bigg)
\end{equation}
where $o(1)$ denotes a function converging to zero when $u\to 0$.
Since we are assuming that $f^{(m+1)}(0)\neq 0$, the estimate~\eqref{second-der-bound-2}
immediately follows from~\eqref{der-3-f-taylor}
taking $\overline\sigma>0$ sufficiently small.
\end{proof}
\medskip

{\bf Proof of lower bounds~\eqref{lest-c}, \eqref{lest-nc} of Theorems~\ref{ThmLUBC}-\ref{ThmLUBNC}}\\

Given any $L, h>0$, recalling definitions~\eqref{Bpm}, \eqref{A+--def}, we have
\begin{equation}
\label{A-B-incl}
\begin{aligned}
\mathcal{A}^+_{[L,\,h]} &\supseteq~ \mathcal{B}_{\big[\!\frac{L}{2},\, \frac{h}{2},\, \leq b_h^+\big]} + \frac{h}{2}\cdot 
\chi_{\big[-\frac{L}{2},\,\frac{L}{2}\big]} 
\qquad\quad \text{if}\qquad\quad f''(h)>0\,,
\\
\noalign{\medskip}
 \mathcal{A}^+_{[L,\,h]} &\supseteq~ \mathcal{B}_{\big[\!\frac{L}{2},\, \frac{h}{2},\, \geq -b_h^+\big]}  + \frac{h}{2}\cdot \chi_{\big[-\frac{L}{2},\,\frac{L}{2}\big]} 
\quad\quad\ \text{if}\qquad\quad f''(h)<0\,,
 \\
\noalign{\medskip}
 \mathcal{A}^-_{[L,\,h]} &\supseteq~ \mathcal{B}_{\big[\!\frac{L}{2},\, \frac{h}{2},\, \leq b_h^+\big]} - \frac{h}{2}\cdot 
 \chi_{\big[-\frac{L}{2},\,\frac{L}{2}\big]} 
\qquad\quad \text{if}\qquad\quad f''(-h)>0\,,
\\
\noalign{\medskip}
 \mathcal{A}^-_{[L,\,h]} &\supseteq~ \mathcal{B}_{\big[\!\frac{L}{2},\, \frac{h}{2},\, \geq -b_h^+\big]}  - \frac{h}{2}\cdot \chi_{\big[-\frac{L}{2},\,\frac{L}{2}\big]} 
\quad\quad\ \text{if}\qquad\quad f''(-h)<0\,.
\end{aligned}
\end{equation}
To fix the ideas, assume now that 
$$f''(h)~>~0,\qquad\qquad f''(-h)~<~0.$$ 
The cases where
$f''(h)>0$, $f''(-h)>0$; $f''(h)<0$, $f''(-h)>0$; or $f''(h)<0$, $f''(-h)<0$, can be treated
in an entirely similar way.
Then, by virtue of Lemma~\ref{inc2}
and relying on~\eqref{A-B-incl}, we find
\begin{equation}
\label{entr-lbound-1}
\begin{aligned}
&\mathcal{H}_{\varepsilon}\Big(S_T({\mathcal C}_{[L,M]}) \ | \ {\bf L}^{1}(\R)\Big)~\geq~
\max\bigg\{
\mathcal{H}_{\varepsilon}\Big(\mathcal{A}^+_{[L,\,h]}  \ | \ {\bf L}^{1}(\R)\Big),\,
\mathcal{H}_{\varepsilon}\Big(\mathcal{A}^-_{[L,\,h]}  \ | \ {\bf L}^{1}(\R)\Big)
\bigg\}
\\
\noalign{\medskip}
&\geq~
\max\bigg\{
\mathcal{H}_{\varepsilon}\Big( \mathcal{B}_{\big[\!\frac{L}{2},\, \frac{h}{2},\, \leq b_h^+\big]} 
+ \frac{h}{2}\!\cdot\! \chi_{\big[-\frac{L}{2},\,\frac{L}{2}\big]}
 \ | \ {\bf L}^{1}(\R)\Big),\,
\mathcal{H}_{\varepsilon}\Big(\mathcal{B}_{\big[\!\frac{L}{2},\, \frac{h}{2},\, \geq -b_h^+\big]} 
- \frac{h}{2}\!\cdot\! \chi_{\big[-\frac{L}{2},\,\frac{L}{2}\big]} 
\ | \ {\bf L}^{1}(\R)\Big)
\bigg\}
\\
\noalign{\medskip}
& =~
\max\bigg\{
\mathcal{H}_{\varepsilon}\Big( \mathcal{B}_{\big[\!\frac{L}{2},\, \frac{h}{2},\, \leq b_h^+\big]} 
 \ | \ {\bf L}^{1}(\R)\Big),\,
\mathcal{H}_{\varepsilon}\Big(\mathcal{B}_{\big[\!\frac{L}{2},\, \frac{h}{2},\, \geq -b_h^+\big]} 
\ | \ {\bf L}^{1}(\R)\Big)
\bigg\}
\end{aligned}
\end{equation}
for all $h>0$ satisfying~\eqref{h-controll-bound}.
Hence, invoking Lemma~\ref{lowb} and because of~\eqref{bh+--def}, we derive from~\eqref{entr-lbound-1}
the estimate
\begin{equation}
\label{entr-lbound-2}
\mathcal{H}_{\varepsilon}\Big(S_T({\mathcal C}_{[L,M]}) \ | \ {\bf L}^{1}(\R)\Big) \geq
\frac{L^2}{108\, \ln 2\cdot T} \cdot \frac{1}{\min \Big\{\displaystyle{\max_{z\in [0, h]}} |f''(z)|,\,\displaystyle{\max_{z\in [-h, 0]}} |f''(z)|\Big\}}
\cdot \frac{1}{\varepsilon}
\end{equation}
for any $h \geq \frac{6\varepsilon}{L}$ such that~\eqref{h-controll-bound} holds.
Choosing $h=\frac{6\varepsilon}{L}$, we recover from~\eqref{entr-lbound-2} the estimate~\eqref{lest-c}
for all $\varepsilon>0$ such that 
\begin{equation}
\label{h-controll-bound-2}
\max_{|z|\leq \frac{6\varepsilon}{L}}~|f'(z)|~\leq~{L\over 2T}\,.
\end{equation}
On the other hand, in the case where $f$ is a nonconvex flux satisfying
conditions~\eqref{EqZeroSpeed}, \eqref{hyp-NC}, applying Lemma~\ref{up-bounds-second-der-f}
and taking $h=\frac{6\varepsilon}{L}$,
we derive from~\eqref{entr-lbound-2} the estimate
\begin{equation}
\label{entr-lbound-3}
\mathcal{H}_{\varepsilon}\Big(S_T({\mathcal C}_{[L,M]}) \ | \ {\bf L}^{1}(\R)\Big)~\geq~
\frac{L^{m+1}}{108\, \ln 2\cdot 6^{m-1}\cdot\overline\alpha \cdot T} \cdot \frac{1}{\varepsilon^m}
\end{equation}
which proves~\eqref{lest-nc}. 

\qed

\end{document}